\documentclass{article}
\usepackage[latin1]{inputenc}
\usepackage[hmargin=1in, vmargin=1in]{geometry}
\usepackage{color,amsmath,amsfonts,amssymb,amsthm,enumerate,graphicx,multicol,multirow}
\usepackage[colorlinks]{hyperref}
\usepackage{epstopdf}

\newtheorem{thm}{Theorem}
\newtheorem{lem}{Lemma}
\newtheorem{cor}{Corollary}

\newtheorem*{defn}{Definition}

\newtheorem{cond}{Condition}

\def\R{\mathbb{R}}
\def\Re{\mathcal{R}}
\def\E{\mathbb{E}}

\def\Z{\mathbb{Z}}

\def\1{\textnormal{\textbf{1}}}

\def\ddt{\frac{d}{d\theta}}

\def\ppt{\frac{\partial}{\partial\theta}}
\def\ppti{\frac{\partial}{\partial\theta_i}}
\def\th{^{\textrm{th}}}
\def\st{^{\textrm{st}}}

\def\xtt{X_\theta(t)}
\def\xts{X_\theta(s)}
\def\xtheta{X_\theta}
\def\ztt{Z_\theta(t)}
\def\zts{Z_\theta(s)}
\def\ztheta{Z_\theta}

\title{Hybrid Pathwise Sensitivity Methods
for Discrete Stochastic  Models of Chemical Reaction Systems}
\author{Elizabeth Skubak Wolf\footnote{Saint Mary's College, ewolf@saintmarys.edu.}, \ David F. Anderson\footnote{University of Wisconsin at Madison, anderson@math.wisc.edu.}}

\begin{document}

\maketitle

\begin{abstract}

Stochastic  models are often 
used to help understand the behavior of intracellular biochemical processes.  The most common such models are continuous time Markov chains (CTMCs).  
Parametric sensitivities, which are derivatives of expectations of model output quantities with respect to model parameters, are useful in this setting for a variety of applications.  
 In this paper,  we introduce a class of hybrid pathwise differentiation methods for the numerical estimation of parametric sensitivities.  The new hybrid methods combine elements from the three main classes of procedures for sensitivity estimation, and have a number of desirable qualities.  First, the new methods are unbiased for a broad class of problems.  Second, the methods are applicable to nearly any physically relevant biochemical CTMC model.  Third, and as we demonstrate on several numerical examples, the new methods are  quite efficient, particularly if one wishes to estimate the full gradient of parametric sensitivities.  The methods are rather intuitive and utilize the multilevel Monte Carlo philosophy of splitting an expectation into separate parts and handling each in an efficient manner.
\end{abstract}

\section{Introduction}

New methods for the estimation of parametric sensitivities are introduced that are applicable to a class of stochastic models widely utilized in the biosciences.  In particular, the theoretical analysis and algorithms provided here extend the validity of the pathwise method developed by Sheppard, Rathinam, and Khammash \cite{Khammash2012}, with related earlier work by Glasserman  \cite{Glasserman1990}, to nearly all physically relevant stochastic models from biochemistry.  The extension is achieved by providing an explicit, numerically computable term for the bias introduced by standard pathwise differentiation methods.  The methods developed here are naturally unbiased and are relatively easy to implement.  Furthermore, they are quite efficient, in some cases providing a speed up of multiple orders of magnitude over the previous state of the art.

\subsection{Mathematical model and problem statement}
\textbf{Mathematical model.} We consider the parametrized family of continuous time Markov chain (CTMC) models satisfying the stochastic equation
\begin{equation}\label{eq:pmodel}
\xtt=X_\theta(0)+\sum_{k=1}^K Y_k\left(\int_0^t
\lambda_k(\theta,\xts)\, ds\right)\zeta_k,
\end{equation} 
where the state space $\mathcal{S}$ of $\xtheta$ is a subset of  $\Z^d$, $K < \infty$, the $\{Y_k\}$ are independent unit-rate Poisson processes, $\theta\in \mathbb{R}^R$ is a vector of model parameters, and where for each $k\in \{1,\dots,K\}$ we have a fixed reaction vector $\zeta_k\in\Z^d$ and a nonnegative intensity, or propensity, function $\lambda_k:\R^R\times\Z^d \to \R_{\geq 0}$.   Such models are used extensively in the study of biochemical processes \cite{AndersonKurtzBook, AndersonKurtz2011, Elowitz, Gillespie, Kurtz78, Paulsson2004, Raj, Wilkinson2011} in which case the vectors $\zeta_k$ can be  decomposed into the difference between the \textit{source vector} $\nu_k \in \Z^d_{\ge 0}$, giving the numbers of molecules required for a given reaction to proceed, and the \textit{product vector} $\nu_k'\in \Z^d_{\ge 0}$, giving the numbers of molecules produced by a given reaction.  Specifically, in this case $\zeta_k = \nu_k' - \nu_k$.   Under the assumption of mass action kinetics, which assumes intensities of the form
\begin{equation}\label{eq:mass-action}
	\lambda_k(\theta,x) = \theta_k \prod_{i = 1}^d \frac{x_i!}{(x_i-\nu_{ki})!}1_{\{x - \nu_{ki} \ge 0\}}, \quad \text{for } x \in \Z^d_{\ge 0},
\end{equation}
 the parameter vector $\theta$ commonly represents some subset of the rate constants $\{\theta_k\}$ of the $K$ reactions.  Note that in the biochemical setting the state space $\mathcal{S}$ is a subset of $\Z^d_{\ge 0}$.

Models of the form \eqref{eq:pmodel}  satisfy the Kolmogorov forward equation, which is typically called the chemical master equation in the biology and chemistry literature,
\begin{equation}\label{eq:cme}
	\frac{d}{dt} p_{\pi}^\theta(t,x) = \sum_{k=1}^K p_{\pi}^\theta (t,x-\zeta_k) \lambda_k(\theta,x-\zeta_k)1_{\{x-\zeta_k \in \mathcal S\}}- \sum_{k=1}^K p_{\pi}^\theta (t,x)\lambda_k(\theta,x),
\end{equation}
where $p_\pi^\theta(t,x)$ is the probability the state of the system is $x\in \mathcal S$ at time $t\ge 0$ given an initial distribution of $\pi$.  
 The infinitesimal generator for the CTMC \eqref{eq:pmodel} is 
 the operator $\mathcal{A}^\theta$ defined via
\begin{equation}\label{eq:gen} (\mathcal{A}^\theta f)(x) = \sum_{k=1}^K \lambda_k(\theta,x)(f(x+\zeta_k)-f(x)),
\end{equation}
for  $f:\mathbb{R}^d\rightarrow\mathbb{R}$ vanishing off a finite set \cite{EthierKurtz}.
For more background on this model see \cite{AndersonKurtzBook, AndersonKurtz2011,Kurtz78,Kurtz1980}. 

    We note that many lattice-valued processes can be represented similarly to \eqref{eq:pmodel}, where a counting process is used to determine the number of jumps that have taken place in one of finitely many specified directions.
    In particular, models satisfying  \eqref{eq:pmodel} also arise in queueing theory and the study of population processes.  
    As biochemical reaction networks are the main motivation for this work, we use biochemical terminology and examples throughout, and simply note that the presented methods are also applicable in those other settings.  

\vspace{.125in}

\noindent \textbf{Problem statement.}
The process $\xtheta$ satisfying \eqref{eq:pmodel} is right continuous and has left hand limits.  That is, $\xtheta$ is c\`adl\`ag and is an element of the Skorohod space $D_{\Z^d}[0,\infty)$.     Consider the output quantity of the CTMC model \eqref{eq:pmodel} given by $ \E [f(\theta, \xtheta)]$,
	  where $f:\R^R\times  D_{\Z^d}[0,\infty) \to \R$ is some measurable function of $\theta$ and $\xtheta$.  
	 We are interested in the problem of numerically computing the gradient $\nabla_\theta \E[f(\theta,\xtheta)]$ for a wide class of functionals $f$.  
	Specifically, we are interested in functionals of the form
	  \begin{equation}\label{eq:setting1}
	  	f(\theta,X_\theta) = h(\theta,X_{\theta}(t)), \quad \text{for $t$ fixed},
	  \end{equation}
	  where $h:\R^R \times \Z^d \to \R,$
	  or path functionals of the form 
	  \begin{align}\label{eq:409857}
	  	L(\theta) := \int_a^b F(\theta,\xts)\, ds,
	\end{align}
	where $0\le a\le b<\infty $ and $F: \R^R\times \Z^d \to \R$.   We will  write $L_X(\theta)$ for $L(\theta)$ when we wish to be clear about the underlying process $X$, and will denote $J(\theta) := \E[L(\theta)]$.

	We will focus most of our attention on functionals of the form \eqref{eq:409857} as we will show in Section \ref{sec:smoothing} how basic smoothing procedures allow us to use such functionals in conjunction with our new methods to provide estimates for $\nabla_\theta \E[f(\theta,X_\theta)]$ when  $f$ is of the form \eqref{eq:setting1}.  
	 Thus,  under some mild regularity conditions on the  functions $\lambda_k$ and $F$ (see Conditions \ref{lineargrowth} and \ref{Fcond} in this section below), we focus on the problem of estimating 
	\begin{equation}\label{eq:sens}
	\nabla_\theta J(\theta) = \nabla_{\theta} \E[ L(\theta)] = \left[ \ppti \E \left(\int_a^b F(\theta,\xts)\, ds\right)\right]_{i=1,\dots,R}
	\end{equation}	
	at some  fixed value $\theta_0 \in \R^R$.  We will generally write $\theta$ rather than $\theta_0$ if the context is clear.

\subsection{A brief review of methods}
\label{sec:1:methods}

Due to the importance of having reliable numerical estimators for  gradients, there has recently been a plethora of research articles focusing on their development and analysis \cite{Anderson2012,AK2014,GuptaKhammash,GK2014,munsky,PlyasunovArkin, Khammash2010,Khammash2012,RawlingsSrivAnd}.   
There are three main classes of methods that carry out the  task of estimating these derivatives: finite difference methods, likelihood ratio methods, and pathwise methods.  Each class has its own benefits and drawbacks.

\begin{itemize}
\item[--] Estimators built via \textbf{finite differences} are easy to implement and  often have a low variance.  However, these estimators provide a  biased estimate \cite{Anderson2012,Khammash2010,RawlingsSrivAnd}.  See Section \ref{sec:FD}.
\item[--] Estimators built using \textbf{likelihood ratios} are unbiased, but often have a high variance \cite{Anderson2012, PlyasunovArkin}.  The use of the usual weight function as a control variate can lower the variance, sometimes dramatically so.  See Section \ref{sec:lrest}.
\item[--] \textbf{Pathwise methods}, known as (infinitesimal) perturbation analysis in the discrete event systems literature \cite{Glasserman1990, GongHo}, are unbiased and are often quite fast \cite{Khammash2012}.  Unfortunately, biochemical models only rarely satisfy the conditions required for the applicability of these methods.  See, for example, the appendix of \cite{Khammash2012} and Section \ref{sec:pathwise} below.    Greatly expanding  the applicability of the pathwise methods already developed for biochemical processes, for example in \cite{Khammash2012}, is one of the main contributions of this work.
\end{itemize}

In some recent works Gupta and Khammash  have developed a new type of  method that does not fit neatly into one of the above categories \cite{GuptaKhammash,GK2014b}.  
Their new method, the Poisson path approximation (PPA) method, which is an improvement on their auxiliary path approximation (APA) method introduced in \cite{GuptaKhammash},  is unbiased and is quite efficient  \cite{GK2014b}.   This method does, however, require additional simulation of partial paths, which may significantly reduce efficiency on some models.

\subsection{A high level overview of the present work}

Elements from each of the three general classes of methods outlined in Section \ref{sec:1:methods} above will be  utilized in the development of  estimators that combine the strengths of each.  Further, the methods introduced here utilize the  multilevel Monte Carlo philosophy by splitting a desired quantity into pieces, and then handling each piece separately and efficiently \cite{AH2012,Giles2008}. Specifically, much of the computational work is carried out with a pathwise method \cite{Khammash2012} applied to an approximate process, ensuring the overall method is efficient.  In order to correct for the bias introduced by the use of an approximate process, the gradient of an error term is computed.   The error term is represented as the expectation of a function of a coupling between the original process and the approximate process.  The likelihood ratio method is used to compute the necessary derivative on this error term.  The coupling used between the exact and approximate processes is the split coupling \cite{Anderson2012,AK2014}.    

Expanding upon the previous paragraph, we give a high level summary of the new method as applied to the functional $L_X(\theta)$ in \eqref{eq:409857}.  First note that by adding and subtracting the appropriate terms,
\[
	\E\left[ \int_a^b F(\theta,\xts) \, ds\right]  = \E\left[\int_a^b (F(\theta,\xts) \, ds - F(\theta,\zts)) \, ds\right] +\E\left[\int_a^b F(\theta,\zts) \, ds\right],
\]
where $Z_\theta$ is any process that can be built on the same probability space as $\xtheta$, and where we assume the expectations above are finite.  Then, assuming the derivatives exist,
\begin{align}\label{eq:2408957}
\nabla_\theta \E\left[ \int_a^b F(\theta,\xts) \, ds\right]  = \nabla_\theta\E\left[\int_a^b (F(\theta,\xts) \, ds - F(\theta,\zts)) \, ds\right] +\nabla_\theta\E\left[\int_a^b F(\theta,\zts) \, ds\right].
\end{align}
We are now able to use different methods to compute the two derivatives on the right-hand side of the above equation.  We have complete control over $Z_\theta$, and we will construct it so that (i)  pathwise methods may be utilized for the final derivative on the right-hand side of \eqref{eq:2408957}, and (ii) $\ztheta$ is a good approximation to $\xtheta$.  The error term, which is the first term on the right-hand side of \eqref{eq:2408957}, will be estimated via a likelihood ratio method.  The efficiency of the overall method  rests upon two facts.  First, the error term can be quickly estimated because its variance will be small if  $Z_\theta$ is a good approximation to $\xtheta$.   This helps overcome the often problematically large variance of a likelihood estimator. Second, the final derivative can be estimated quickly because pathwise methods are fast when they are applicable.  

In this paper we present what we believe is a reasonable choice for the process $\ztheta$ in \eqref{eq:2408957}.  Specifically, it will have the same jump directions $\{\zeta_k\}$ as $\xtheta$, but different intensity functions and an enlarged state space.  While we hope to impart why we believe it to be a good choice, many other options for $\ztheta$ exist and can be explored in future research.  Improvements in the selection of the process $\ztheta$ will correspond with improvements to the overall method.  The use of multilevel Monte Carlo on either of the needed derivatives could also lead to a significant improvement in efficiency.

Our numerical examples section shows that the methods we introduce here fit well into the group of existing  methods  for the numerical estimation of parametric sensitivities in the jump process setting.  They are quite efficient on all examples, sometimes significantly more efficient than any other existing method.  However, and not surprisingly given the amount of effort that has been put into development over the past few years, they are not \textit{always} the most efficient.  In particular, sometimes PPA (Gupta and Khammash, \cite{GK2014b}) or the coupled finite difference method (Anderson, \cite{Anderson2012})  is most efficient.   With such a strong group of methods having been developed over the past few years, we feel future work in the field should also include the determination of which methods to use  for different model and problem types.

\subsection{Regularity conditions}

We end this introduction with two  regularity conditions which are necessary for the validity of the methods introduced here.  The first condition guarantees that solutions to equation \eqref{eq:pmodel} exist for all time.  The second condition relates to $F$ of \eqref{eq:409857}, and simply ensures that $F$ does not grow too fast in the $x$ variable.  Both conditions are required in our proofs in Appendix \ref{app:PE}.
Conditions to be satisfied by the approximate process $Z_\theta$ will be developed as needed throughout the paper. In particular, see Conditions \ref{nonint}, \ref{cond:extra}, and \ref{coupledlrcond}.

For $x\in \Z^d$ we use the notation $\Vert x \Vert$  to denote the 1-norm, $\Vert x \Vert = \sum_{i=1}^d|x_i|$. 

\begin{defn}
We say that $h: \R^R \times \mathcal{S} \to \R$ has \textbf{uniform polynomial growth at $\theta$} if there is a neighborhood $\Theta \subset \R^R$ of $\theta$ and constants $C,p >0$ such that
$\big|\sup_{\theta\in\Theta} h(\theta,x)\big| \leq C(1+\Vert x\Vert^p)$ for all $x\in\mathcal{S}$.
 If $p$ may be taken to be 1, we say that $h$ has \textbf{uniform linear growth at $\theta$}.

\end{defn}
\noindent 
Let $\mathbf{1}$ denote the vector of all ones.  Define $\Re_1\subset \{1,\dots,K\}$ so that $k \in \Re_1$ if and only if $\mathbf{1}\cdot \zeta_k > 0$.  Define  $\Re_2=\{1,\dots,K\}\setminus\Re_1$.  Note that if $\mathcal S \subset \Z^d_{\ge 0}$, then  $\Re_1$ contains the indices of those reactions that increase the total population, i.e.
\[
	 \Vert x + \zeta_k \Vert > \Vert x\Vert, \quad \text{ for all } x \in \mathcal S,
\]
 while reactions with indices in $\Re_2$ either decrease the total population or leave it unchanged.

\begin{cond}\label{lineargrowth}
The intensities $\lambda_k$ satisfy this condition at $\theta$ if there is some neighborhood $\Theta \subset \R^R$ of $\theta$ such that:
\begin{enumerate}
\item for each $k\in\{1,\dots,K\}$ and  $\theta \in \Theta$, the function $\lambda_k$ has uniform polynomial growth at $\theta$;
\item for each $k\in\{1,\dots,K\}$, $i\in\{1,\dots,R\}$, and $\theta \in \Theta$, the function $\ppti\lambda_k$ exists and has uniform polynomial growth at $\theta$;
\item for each $k\in\Re_1$ and $\theta \in \Theta$, the function $\lambda_k$ has uniform linear growth at $\theta$;
\item there exist constants $p$ and $C$  such that for all $k\in \{1,\dots,R\}$ and all $x\in\mathcal{S}$
\[
	\sup_{\theta\in\Theta}\lambda_k(\theta,x) \neq 0 \Rightarrow \sup_{\theta\in\Theta}\frac{1}{\lambda_k(\theta,x)} \leq C(1+\Vert x \Vert^{p});\]
that is, for a fixed $x$, if the rates $\lambda_k(\theta,x)$ are not identically zero on  $\Theta$, then they must be bounded away from zero. 
\end{enumerate}
\end{cond}

\noindent 
Note that the third part of Condition \ref{lineargrowth}, which requires certain intensities to grow at most linearly, only applies to those intensity functions that are  associated with  transitions that increase the total population count of the system.   Essentially, this portion of Condition \ref{lineargrowth} ensures that the population does not explode in finite time, and could almost certainly be weakened.     We note that this condition was also utilized in \cite{GuptaKhammash}.  Condition \ref{lineargrowth} is satisfied for most  biochemical systems considered in the literature.  In particular, it is satisfied by any binary chemical system with mass action kinetics.\footnote{A chemical system is binary if $\sum_{i=1}^d |\nu_{ki}| \le 2$ and $\sum_{i=1}^d |\nu_{ki}'|\le 2$ for each $k \in \{1,\dots,K\}$.} 
  For example, assuming mass action kinetics, the reactions $A \to 2A$ and $2A \to B+C$ are permissible under Condition \ref{lineargrowth}.
On the other hand, Condition \ref{lineargrowth} excludes $2A \to 3A$, which increases the population at a quadratic rate, and can lead to explosions.

We turn to the  regularity conditions for $F$ of \eqref{eq:409857}.  The following condition will allow us to bound moments of $L$ using the moments of the process $X_\theta$.
\begin{cond}\label{Fcond}
Let $\Theta \subset \R^R$.  The function $F:\Theta\times \mathcal{S} \to \R$ satisfies this condition if it is measurable, and differentiable in $\theta$ on $\Theta$ so that:
\begin{enumerate}
\item
there exist constants $C_A>1$ and $p_A>1$ such that  $\displaystyle\sup_{\theta\in\Theta}
|F(\theta,x)| \leq C_A(1+\Vert x \Vert^{p_A})$ for all $x \in \mathcal{S}$;

\item there exist constants $C_B>1$ and $p_B>1$ such that for all $i\in\{1,\dots,R\}$ and $x \in \mathcal{S}$ we have\\ $\displaystyle\sup_{\theta\in\Theta}\bigg|\ppti F(\theta,x)\bigg| \leq C_B(1+\Vert x \Vert^{p_B}).
$

\end{enumerate}\end{cond}

%
%
%

The outline for the remainder of the paper is as follows.  In Section \ref{sec:pwbackground}, we introduce the three main classes of methods for the numerical estimation of parametric sensitivities.  In particular, in Section \ref{sec:pathwise} we present Thereom \ref{thm:pw_first}, which gives conditions for the validity of pathwise methods for functionals of the form \eqref{eq:409857}.  In Section \ref{sec:hyb}, we introduce an approximate process $\ztheta$ to be utilized in  \eqref{eq:2408957} and formally present the new methods.
In Section \ref{sec:hybex}, we demonstrate several numerical results, and we present conclusions in Section \ref{sec:concl}.  The proof of Theorem \ref{thm:pw_first} is included in the appendix.
	

\section{Classes of Methods}\label{sec:pwbackground}

We introduce the three main classes of methods for the numerical estimation of parametric sensitivities:  finite differences, pathwise derivatives, and likelihood ratios.   Because the methods introduced here involve both pathwise derivatives and likelihood ratios, we discuss both in detail in Sections \ref{sec:pathwise} and \ref{sec:lrest} below.  Throughout these sections, we also introduce and motivate the regularity conditions and theoretical results that are required for the approximate process $Z_\theta$ of \eqref{eq:2408957}.  In Section \ref{sec:hyb}, we will combine these pieces to succinctly introduce our new method.
Our main theoretical results pertaining to pathwise methods are stated in Section \ref{validity} and proven in the appendix. 

\subsection{Finite differences}\label{sec:FD}

Let $e_i\in \R^R$ be the vector of all zeros except a one in the $i$th component.  Finite difference methods proceed by simply noting that for  $f: \R^R \times D_{\Z^d}[0,\infty)\to \R$,
\begin{align*}
	\ppti \E[f(\theta,\xtheta)]  &\approx  h^{-1} \left( \E[f(\theta+he_i,X_{\theta + he_i})]-  \E[f(\theta,\xtheta)]\right)\\
	&= h^{-1} \E [ f(\theta+he_i,X_{\theta + he_i})- f(\theta,\xtheta)],
\end{align*}
as long as the derivatives and expectations exist, and where the final equality implies  the two processes have been built on the same probability space, or \textit{coupled}. The coupling is used in order to reduce the variance of the difference between the two random variables.   The two most useful couplings  in the present context are the common reaction path method \cite{Khammash2010} and the split coupling method \cite{Anderson2012}, the latter of which we detail explicitly in Section \ref{sec:lrest} in and around \eqref{eq:coup}.

\subsection{Pathwise methods}\label{sec:pathwise}
When using a pathwise method, one begins with a probability space that does not depend on $\theta$;  instead, one uses $\theta$ to construct the path from the underlying randomness.  For our purposes, we take a filtered probability space $(\Omega, \mathcal F, \{\mathcal{F}_t\}_{t\geq 0},Q)$ under which $\{Y_k, k=1,\dots,K\}$ are independent unit-rate Poisson processes. 
The path $X_\theta$ is then constructed by a jump by jump procedure implied by \eqref{eq:pmodel}, which is equivalent to an implementation of the next reaction method \cite{Anderson2007, Gibson2000}.
For ease of exposition, we restrict ourselves to consideration of one element of the gradient, $\ppti J(\theta)$, though calculation of the full gradient can be carried out in the obvious manner.  

Consider a general functional $f$.  If the following equality holds,
\begin{equation}\label{eq:249085}
\frac{\partial}{\partial \theta_i} \E[ f(\theta,\xtheta)] = \E\left [  \frac{\partial}{\partial \theta_i}f(\theta,\xtheta)\right],
\end{equation}
then   $\frac{\partial}{\partial \theta_i}  \E[ f(\theta,\xtheta)] $ can be estimated via Monte Carlo by repeated sampling of independent copies of the random variable $\frac{\partial}{\partial \theta_i}f(\theta,\xtheta)$.
Unfortunately, for a wide variety of models of the form \eqref{eq:pmodel} and functionals $f$, equality in  \eqref{eq:249085} does not hold.  There are typically two reasons for this.
\begin{enumerate}[1.]
\item In many cases the random variable $\frac{\partial}{\partial \theta_i}f(\theta,\xtheta)$ is almost surely zero, in which case the right hand side of \eqref{eq:249085} is zero whereas the left hand side is not.
\item The underlying process $X_\theta$ can undergo an \textit{interruption,} in which case  $\E\left [  \frac{\partial}{\partial \theta_i}f(\theta,\xtheta)\right]$ is typically non-zero, but still not equal to $\frac{\partial}{\partial \theta_i} \E[ f(\theta,\xtheta)]$.
\end{enumerate}

The first problem stated above commonly arises when $f$ is a function solely of the process at the terminal time $T$,  i.e. when $f(\theta,X_\theta) = h(X_\theta(T))$ for some $T>0$ and $h:\mathcal S \to \R$ (as in \eqref{eq:setting1} above).  Then, since $\xtheta$ is a CTMC and  has piecewise constant paths,
$ \frac{\partial}{\partial \theta_i}h(X_{\theta}(T)) = 0$ almost surely.    This type of problem is easily overcome by any number of smoothing procedures, with a few outlined below in Section \ref{sec:smoothing}.  The second problem, in which there is an interruption, is a more serious problem with the method.  Interruptions are discussed in more detail in Section \ref{sec:non-interuptive} below.  Overcoming this type of problem while still utilizing the pathwise framework can be viewed as a major contribution of this work.

\subsubsection{Smoothing}
\label{sec:smoothing}

As will be seen in Section \ref{sec:pwest},  pathwise methods are capable of providing  estimates of derivatives of functionals of the form $\int_a^b F(\theta,\xts) \, ds,$
where $a, b \in \R$ and $F: \R^R\times \Z^d \to \R$ satisfies mild regularity conditions.
Thus, in order to estimate derivatives of, for example,  $\E[f(X_\theta(T))]$, where $f:\mathcal S \to \R$, one simply needs to replace $f(X_{\theta}(T))$ with an appropriate integral.  There are a number of natural choices, with only a few discussed here.

The Regularized Pathwise Derivative (RPD) method presented in \cite{Khammash2012} estimates $\ppti \E [f(X_\theta(T))]$ using independent copies of $\theta$-derivatives of 
\begin{equation}\label{eq:rpd}
L_1(\theta) := \frac{1}{2w}\int_{T-w}^{T+w} f(\xts)ds \approx f(X_\theta(T)),
\end{equation}
where $w$ is some fixed window size.
Note that even when pathwise methods can be applied to the model, i.e. when there are no interruptions, this method gives a biased estimate, with the size of the bias a function of the size of $w$.  Specifically, a smaller $w$  leads to a smaller bias but a larger variance.

Alternatively, one may martingale methods to derive an unbiased estimator.  
Specifically,  for a large set of functions $f:\Z^d \to \R$,
\begin{equation}\label{eq:mart_rep}
	f(\xtt) = f(X_\theta(0)) + \int_0^t ({\mathcal A}^\theta f) (\xts)\, ds + M_t^\theta,
\end{equation}
where $M_t^\theta$ is a local martingale and $\mathcal{A}^\theta$ is the generator \eqref{eq:gen} \cite{AndersonKurtzBook,EthierKurtz}.  In many cases of interest $M_t^\theta$ is a martingale, in which case \eqref{eq:gen} implies
\begin{align}\begin{split}\label{eq:dynkin}
\E [f(\xtt)] 
& = \E \left[ f(X_\theta(0)) + \int_0^t \sum_{k} \lambda_k(\theta,\xts)[f(\xts+\zeta_k) - f(\xts)] \, ds\right].
\end{split}\end{align}
For example, for processes $\xtheta$ that satisfy Condition \ref{lineargrowth}, which is nearly all biologically relevant processes,  equation \eqref{eq:dynkin} holds for functions $f$ that grow at most polynomially.
Therefore, another option for a smoothing functional would be to take 
\begin{equation}\label{eq:34208957}
	L_2(\theta) := f(X_\theta(0))+\int_0^T (\mathcal{A}^\theta f)(X_\theta(s))ds
\end{equation}
in which case $\ppti \E [f(X_\theta(T))] = \E[ \frac{\partial}{\partial \theta_i} L_2(\theta)]$ (see also \cite{Glasserman1990}, p.~176).
While unbiased when it applies, this estimator tends to have higher variance than the RPD estimator so long as the parameter $w$ is not taken too small.  We shall refer to the smoothing procedure \eqref{eq:34208957} as the Generator Smoothing (GS) method and will refer to the estimation procedure
\[
	\ppti \E[f(X_\theta(T))] \approx \frac{1}{n} \sum_{j = 1}^n \ppti L_2^{[j]}(\theta),
\]
where $L_2^{[j]}(\theta)$ is the $j$th independent realization of the random variable $\ppti L_2(\theta)$, as   as the GS Pathwise method.  An algorithm for the generation of the random variable  $\ppti L_2(\theta)$ is given in Section \ref{sec:pwest} below.

\subsubsection{The non-interruptive condition}
\label{sec:non-interuptive}

 Smoothing alone does not always ensure the validity of a pathwise method: for $L(\theta)$ given by \eqref{eq:409857} we still may have $\ppti \E [L(\theta)] \ne \E [ \ppti L(\theta)]$.  
Again letting $e_i \in \R^R$ be the vector of all zeros except a one in the $i$th component, for $\xtheta$ satisfying Condition \ref{lineargrowth} it is straightforward to show that
\begin{equation}\label{eq:aeconv}
\lim_{h\to 0} \E \left[\frac{L(\theta+he_i) - L(\theta)}{h} \right] = \ppti\E \left[ L(\theta) \right]
\quad \textrm{and} \quad 
\frac{L(\theta+h e_i) - L(\theta)}{h} \overset{a.s.}{\longrightarrow} \ppti L(\theta).
\end{equation}
However, to have the equality
\begin{equation}\label{eq:L1}
\ppti\E \left[ L(\theta) \right]
= \E \left[\ppti L(\theta)
 \right],
\end{equation}
we must have convergence in mean in addition to the a.s.~convergence in \eqref{eq:aeconv}.  The following condition will play a central role in achieving the convergence in mean. A similar condition was first introduced by Glasserman in the discrete event simulation literature \cite{Glasserman1990}.   Recall that $\mathcal{S}$ is the state space of our process.

\begin{cond}[Non-Interruptive]\label{nonint}
The functions $\lambda_k:\Theta \times \mathcal S \to \R_{\ge 0}$, for $k \in\{1,\dots,K\}$, satisfy this condition if for each $k,\ell\in\{1,\dots,K\}$, $x \in \mathcal S$, and  $\theta\in\Theta$, the following holds:  if $\lambda_k(\theta,x)>0$ and $\lambda_\ell(\theta,x)>0$ for $\ell\neq k$, then 
 $\lambda_\ell(\theta,x+\zeta_k)>0$.
\end{cond}

In accordance with terminology from the discrete event simulation literature, we define an  \textit{interruption} as a change in state, from $x$ to $x+\zeta_k$ for some $k$, such that for some $\ell\neq k$ we have $\lambda_\ell(\theta,x)>0$ and $\lambda_\ell(\theta,x+\zeta_k)=0$.  If an interruption occurs, the function $L(\theta)$ can have a jump discontinuity in $\theta$ for a given realization of the process, and \eqref{eq:L1} can fail to hold. 
The non-interruptive Condition \ref{nonint}, therefore, ensures that interruptions cannot occur. 

Many biological models do not satisfy Condition \ref{nonint}.
For a simple example of a model that does not satisfy the non-interruption condition, consider the  reaction network 
\[
	A \to \emptyset, \quad A \to B,
\]
which has reaction vectors 
\[
	\left[\begin{array}{c}
	-1\\0
	\end{array}\right] \quad \text{and} \quad \left[\begin{array}{c}
	-1\\1
	\end{array}\right].
\]
Endow the system with mass action kinetics and an initial condition of precisely one $A$ particle and zero $B$ particles.  Then the occurrence of either reaction will necessarily cause an interruption. 

For  models in which interruptions are possible, which includes most biochemical models, both the GS pathwise method and the RPD method may produce significant bias when estimating gradients.  See Appendix B of \cite{Khammash2012} for a comment on this issue, and see Section \ref{sec:hybex} below where the bias is demonstrated numerically.

\subsubsection{An algorithm for calculating $\ppti L(\theta)$}
\label{sec:pwest}

Providing realizations of the random variable $\ppti L(\theta)$, where $L$ is of the form \eqref{eq:409857}, is central to the methods presented here.  This section provides the necessary numerical algorithm.    The derivations are based on simulating the random time change representation \eqref{eq:pmodel} using the next reaction method.    Conditions on the intensity functions   guaranteeing that   $\ppti \E[L(\theta)] = \E[ \ppti L(\theta)]$ are provided in  Section \ref{validity} below.

 We note that the algorithm derived within this section is essentially the same as those derived in  \cite{Glasserman1990} and \cite {Khammash2012}.   This section is included for completeness, but can be safely skipped by those familiar with pathwise differentiation.

Recalling the discussion in and around \eqref{eq:2408957}, the methods introduced in this article  use pathwise differentiation on functionals of  a non-interruptive  process.  This process is typically an  \textit{approximation} of the original process.  Thus,  in this section we denote our nominal process by $\ztheta$ as opposed to $\xtheta$. 
 Further, 
for notational convenience in this section we take $\theta$ to be 1-dimensional.

Continuing, we suppose $\ztheta$ is a process satisfying the stochastic equation  \eqref{eq:pmodel} with $\theta \in \R$.  
Let $\hat Z_\ell(\theta)$ denote the $\ell\th$ state in the embedded discrete time chain of the process $\ztheta$, and let $T_\ell^\theta$ be the $\ell\th$ jump time, with $T_0^\theta=0$. 
 We are interested in computing the $\theta$-derivative of 
\begin{equation}\label{L}
L_Z(\theta) := \int_a^bF(\theta,\zts)\, ds
= \sum_{\ell=0}^{N(\theta,b)} F(\theta,\hat Z_\ell(\theta))[ T_{\ell+1}^{\theta}\wedge b -T_{\ell}^{\theta}\vee a ]^+,
\end{equation}
where $a\wedge b=\min(a,b)$ and $a\vee b = \max(a,b)$, and where $N(\theta,b)=N$ is the number of jumps of the process through time $b$. If  $\ztheta$ is a non-explosive process, then $N<\infty$ with a probability of one. 

The embedded chain is discrete-valued. Thus,  $\ppt \hat Z_\ell(\theta)=0$ a.s.~wherever the derivative exists.  Therefore, by \eqref{L},
\begin{align}\begin{split}
\label{pptL}
\ppt L_Z(\theta) = \sum_{\ell=0}^{N} \Bigg[ [ T_{\ell+1}^\theta\wedge& b -T_{\ell}^\theta\vee a ]^+
 \left(\ppt F(\theta,\hat Z_\ell(\theta)) \right) 
+ 
 F(\theta,\hat Z_\ell(\theta))\ppt[ T_{\ell+1}^\theta\wedge b -T_{\ell}^\theta\vee a ]^+ \Bigg],
\end{split}\end{align} 
where the partial of the function $F$ is with respect to the first variable.
The terms involving the derivatives $\ppt F(\theta,\hat{Z}_\ell(\theta))$ are straightforward to compute. The remaining terms require the derivatives of the jump times $T_\ell^\theta$, so we now focus on their derivation.

Define $\Delta_\ell^\theta=T_{\ell+1}^\theta-T_\ell^\theta$ to be the holding time of the process in the $\ell\th$ state (so that the indexing begins at 0). 
 Let $S_k^\theta(t) = \int_0^t \lambda_k(\theta,\zts)ds$.   
Note that $S_k^\theta(t)$ is the argument of the Poisson process $Y_k$ in the stochastic equation \eqref{eq:pmodel}.  The quantity $S_k^\theta(t)$ is therefore usually referred to as the `internal time' of $Y_k$.
 Let 
 \[
 	I_k(t) = \inf\left\{r\geq S_k^\theta(t) : Y_k(r) > Y_k(S_k^\theta(t))\right\} 
\] 
be the internal time of the first occurrence of $Y_k$ after time $S_k(t)$.
 Then the holding time of the process $Z_\theta$ in the $\ell$th state is given by
\begin{equation}\label{eq:9786786}
\Delta_\ell^\theta = \min_k \left\{ \frac{I_k(T_\ell^\theta) - S^\theta_k(T_\ell^\theta)}{ \lambda_k(\theta,\hat Z_\ell(\theta))} \right\}.
\end{equation}
Let $k_\ell$ be the argmin in the above expression; 
$k_\ell$ 
is the index of the reaction that changes the system from the $\ell\th$ to the $(\ell+1)\st$ state. Via the product rule, we have
\begin{align}\begin{split}\label{pptDelta}
\frac{\partial}{\partial\theta}\Delta_\ell^\theta 
&=  -\frac{I_{k_\ell} - S^\theta_{k_\ell}(T_\ell^\theta)}{ \lambda_{k_\ell}(\theta,\hat Z_\ell(\theta))^2}
\ppt \lambda_{k_\ell}(\theta,\hat Z_\ell(\theta)) - \lambda_{k_\ell}(\theta,\hat Z_\ell(\theta))^{-1} \ppt S^\theta_{k_\ell}(T_\ell^\theta)
\\
& =  -\frac{\Delta_\ell^\theta}{ \lambda_{k_\ell}(\theta,\hat Z_\ell(\theta))}
\ppt \lambda_{k_\ell}(\theta,\hat Z_\ell(\theta)) - \lambda_{k_\ell}(\theta,\hat Z_\ell(\theta))^{-1} \ppt S^\theta_{k_\ell}(T_\ell^\theta),
\end{split}\end{align}
where the second equality follows from \eqref{eq:9786786}.
Note that  for $t\in\left[T_\ell^\theta,T_{\ell+1}^\theta\right]$ and any $k \in \{1,\dots,K\}$ we have that 
$ S^\theta_k(t) = S^\theta_k(T_\ell^\theta) + \lambda_k(\theta,\hat Z_\ell(\theta))(t-T_\ell^\theta)$.  Thus
\begin{equation}\label{pptT}
\ppt S^\theta_{k}(T_{\ell}^\theta)= \ppt S^\theta_k(T_{\ell-1}^\theta) + \Delta^\theta_{\ell-1}\ppt \lambda_k(\theta,\hat Z_{\ell-1}(\theta)) + \lambda_k(\theta,\hat Z_{\ell-1}(\theta))\ppt \Delta^\theta_{\ell-1}.
\end{equation}
The values $\{\ppt \Delta^\theta_\ell\}$ and $\{\ppt S^\theta_\ell(T_\ell^\theta)\}$  can now be solved for recursively given that  $S^\theta_{k}(T_{0}^\theta)=0$ for all $k$.  

 To find the derivatives of the $T_\ell^\theta$ as in \eqref{pptL}, first note that $\ppt T_0^\theta = 0$, and that for $\ell>0$ the definition of $\Delta_\ell^\theta$ implies that 
\[
\ppt T_\ell^\theta = \sum_{j=0}^{\ell-1} \ppt \Delta_j^\theta.
\]
Let $\ell_a\in\mathbb{N}$ be maximal such that $T^\theta_{\ell_a}\leq a$; that is, the $\ell_a\th$ jump is the last jump to occur before time $a$.
We may now conclude that
\begin{equation}\label{derivs}
\ppt[  T_{\ell+1}^\theta\wedge b -T_{\ell}^\theta\vee a ]^+ = 
\begin{cases}
 0 & \ell< \ell_a \quad \textrm{or} \quad \ell>N
 \\[1ex]
\ppt T_{\ell_a+1}^\theta= \sum_{j=0}^{\ell_a} \ppt\Delta_j^\theta  & \ell = \ell_a
 \\[1ex]
 \ppt\Delta_\ell^\theta & \ell_a< \ell<N
 \\[1ex]
-\ppt T_N^\theta=-\sum_{j=0}^{N-1} \ppt\Delta_j^\theta  & \ell = N
\end{cases},
\end{equation}
which can all be easily computed during numerical simulation.

The derivations above lead to the following algorithm for the generation of $Z_\theta$ over the interval $[0,b]$ and of the random variable $\ppt L_Z(\theta) = \ppt \int_a^b F(\theta,Z_\theta(s))ds$.  The notation in the algorithm provided below is the same as that above with the following exceptions:
\begin{enumerate}[$i$.)]
\item  \textit{flag} is a variable that only takes the values zero or one.  It starts at zero and becomes one once $t\ge a$.   In the algorithm, this moment is determined by finding the first time at which the process makes a jump at a time greater than $a$ (see Step \ref{step:78968976} below).

\item   The output  $\ppt L_Z(\theta)$, as given in \eqref{pptL}, is denoted by $dL$.  
\end{enumerate}

It may be helpful for the reader to note that steps \ref{step:786586745}, \ref{step:8976}, \ref{step:8766},  \ref{step:6}, \ref{step:56474} and \ref{step:34657} make up the usual implementation of the next reaction method \cite{Anderson2007,Gibson2000}.  Only steps \ref{step:5678}, \ref{step:78968976}, \ref{step:8765675}, and \ref{step:9807789} are those required for  the  derivative terms.
All  uniform random variables generated in the algorithm below are assumed to be mutually independent.
\vspace{.125in}

\noindent \textsc{Algorithm.} Numerical derivation of $Z_\theta$ and   $\ppt L_Z(\theta) =  \ppt \int_a^b F(\theta,Z_\theta(s))\, ds$.

\noindent \textbf{Initialize.}    Given: a continuous time Markov chain with jump directions $\zeta_k$, intensities $\lambda_k(\theta,z)$, and initial condition $z_0$.  Set $\ell = 0$, $T_0^\theta = 0$,  $Z_\theta(T^\theta_0) = z_0$, $\ppt T^\theta_0 =0$, and $dL  = 0$.  For each $k\in\{1,\dots,K\}$, set $S_k^\theta(T^\theta_0) = 0$, $\ppt S_k^\theta(T^\theta_0) = 0$.  Set \textit{flag} = 0.
For each $k\in \{1,\dots,K\}$, set $I_k(T^\theta_0) = \ln(1/u_k)$, where $\{u_k\}$ are independent uniform$(0,1)$ random variables.

\vspace{.05in}
Perform the following steps.
\begin{enumerate}
%

\item \label{step:786586745} For all $k\in \{1,\dots,K\}$, calculate $\lambda_k(Z_\theta(T^\theta_\ell))$. Set 
\[
	\Delta_\ell^\theta = \min_k \frac{I_k(T^\theta_\ell)-S_k^\theta(T^\theta_\ell)}{\lambda_k(\theta,Z_\theta(T^\theta_\ell))} \quad \text{ and } \quad j= \underset{k}{\textrm{argmin}} \frac{I_k(T^\theta_\ell)-S_k^\theta(T^\theta_\ell)}{\lambda_k(\theta,Z_\theta(T^\theta_\ell))}.
\]

\item \label{step:8976}
If $T^\theta_\ell + \Delta_\ell^\theta > b$, go to Step \ref{step:9807789}. Otherwise set $T^\theta_{\ell+1} = T^\theta_\ell + \Delta_\ell^\theta$ and continue to Step \ref{step:5678}.

\item \label{step:5678} Set 
\[
	\ppt \Delta_\ell^\theta = -\frac{\Delta_\ell^\theta}{\lambda_j(\theta,Z_\theta(T^\theta_\ell))}\cdot \ppt\lambda_j(\theta,Z_\theta(T^\theta_\ell)) - \frac{\ppt S_j^\theta(T^\theta_\ell)}{\lambda_j(\theta,Z_\theta(T^\theta_\ell))},
	\]
	then set $\ppt T_{\ell+1}^\theta = \ppt T_\ell^\theta + \ppt \Delta_\ell^\theta$.

\item \label{step:78968976} Set 
\[	
	dL \leftarrow dL + \Delta_\ell^\theta\cdot \ppt F(\theta,Z_\theta(T^\theta_\ell)) + F(\theta,Z_\theta(T^\theta_\ell))\cdot A,
	\]
	 where 
	 \[
	 	A = 
\begin{cases}
0 & \text{if } T^\theta_{\ell+1} < a \\
\ppt T_{\ell+1}^\theta & \text{if } T^\theta_{\ell+1} >a \quad \textrm{and} \quad flag=0\\
\ppt \Delta_\ell^\theta & \textrm{otherwise}
\end{cases}.
\]
 If $T^\theta_{\ell+1}>a$ and $flag=0$, set $flag=1$.

\item \label{step:8766} Set $Z_\theta(T^\theta_{\ell+1}) = Z_\theta(T^\theta_\ell) +\zeta_j$.  

\item \label{step:6} For each $k\in\{1,\dots,K\}$, set $S_k^\theta(T^\theta_{\ell+1}) = S_k^\theta(T^\theta_\ell) + \Delta_\ell^\theta \lambda_k(\theta,Z_\theta(T^\theta_\ell))$.

\item  \label{step:8765675} For each $k\in\{1,\dots,K\}$, set 
\[
	\ppt S_k^\theta(T^\theta_{\ell+1})= \ppt S_k^\theta(T^\theta_{\ell}) + \Delta_\ell^\theta\cdot \ppt\lambda_k(\theta,Z_\theta(T^\theta_\ell))+\lambda_k(\theta,Z_\theta(T^\theta_\ell)) \cdot \ppt \Delta_\ell^\theta.
\]

\item \label{step:56474} Set $I_j(T^\theta_{\ell+1}) = I_j(T^\theta_\ell) +  \ln\left(\frac{1}{u}\right)$, where $u$ is a uniform$(0,1)$ random variable.

\item \label{step:34657} Set $\ell \leftarrow \ell+1$ and return to Step \ref{step:786586745}.

\item \label{step:9807789} Set $dL \leftarrow dL + (b-T^\theta_\ell)\ppt F(\theta,Z_\theta(T^\theta_\ell)) - flag\cdot  F(\theta,Z_\theta(T^\theta_\ell))\cdot \ppt T_\ell^\theta$. 
\end{enumerate}

\subsubsection{Validity of pathwise estimators}\label{validity}


Letting $\ztheta$ be a process satisfying a stochastic equation of the form \eqref{eq:pmodel}, we turn to the  question of when $\ppti \E[L_Z(\theta)] = \E[\ppti L_Z(\theta)]$, with $\ppti L_Z(\theta)$ detailed in the previous section.  For our proof of Theorem \ref{thm:pw_first}, we require a condition on the intensity functions of $Z_\theta$ that is more restrictive than Condition \ref{lineargrowth}.

\begin{cond}\label{cond:extra}
Let $\Theta \subset \R^R$.  The functions $\lambda_k:\Theta\times \Z^d\to \R_{\ge 0}$, $k = 1,\dots, K$,  satisfy this condition if each of the following hold.
\begin{enumerate}
\item There exist constants $\Gamma_M, \Gamma'$ such that for all $k\in \{1,\dots,K\}$ and all $z\in\Z^d$, 
\[
	\sup_{\theta\in\Theta}\sup_{z\in\mathcal{S}}\lambda_k(\theta,z) \le \Gamma_M \quad \text{and} \quad \sup_{\theta\in\Theta}\sup_{z\in\mathcal{S}}\bigg|\ppti\lambda_k(\theta,z)\bigg| \leq 
	\Gamma'.
\]
\item There exists some constant $\Gamma_m$ such that for all $k\in \{1,\dots,K\}$ and all $z\in\Z^d$, 
\[
	\sup_{\theta\in\Theta}\lambda_k(\theta,z) \neq 0 \Rightarrow \sup_{\theta\in\Theta}\frac{1}{\lambda_k(\theta,z)} \leq \Gamma_m.
\]

\end{enumerate}
\end{cond} 
The first condition guarantees  that the intensities and their $\theta$-derivatives are uniformly bounded above.  The second condition stipulates that on those $z\in\Z^d$ at which the rates $\lambda_k(\theta,z)$ are not identically zero on  $\Theta$, the rates must be uniformly bounded away from zero.

\begin{thm}\label{thm:pw_first}
Suppose that the process $\ztheta$ satisfies the stochastic equation \eqref{eq:pmodel} with $\lambda_k$ satisfying Conditions \ref{nonint} and \ref{cond:extra} on a neighborhood $\Theta$ of $\theta$.  Suppose that the function $F$ satisfies Condition \ref{Fcond} on $\Theta$.  For some $0 \le a \le b < \infty$, let  $L_Z(\theta) = \int_a^b F(\theta,\zts)\, ds$. Then  $ \frac{\partial}{\partial \theta_i} \E \left[ L_Z(\theta)\right]= \E\left[ \frac{\partial}{\partial \theta_i} L_Z(\theta)\right]$, for all $i \in \{1,\dots,R\}$.
\end{thm}

The proof of this theorem is similar to that found in \cite{Glasserman1990} and can be found in Appendix \ref{app:PE}.  We believe that the stringent Condition \ref{cond:extra} can be replaced by the more relaxed Condition \ref{lineargrowth}, though this remains open.  The stricter Condition \ref{cond:extra} plays little role in the  methods developed here as it can be incorporated into the definition of the process $\ztheta$, as will be seen in Section \ref{sec:hyb}.  In particular, we note that we will not be requiring that our actual process of interest, $\xtheta$, satisfies Condition \ref{cond:extra}, only that the approximate process, $\ztheta$, does.

\subsection{Likelihood ratios and coupled paths}\label{sec:lrest}

The likelihood ratio (LR) method for sensitivity estimation proceeds by selecting a realization of a given process according to a $\theta$-dependent probability measure.  Differentiation of the probability measure is then carried out within the expectation.
For CTMC models $\xtheta$ as in \eqref{eq:pmodel} that have $\theta$-differentiable intensities and that satisfy the growth Condition \ref{lineargrowth} (which, recall, is nearly all biochemical systems), and for a large class of functionals $f$ we have 
\begin{equation}\label{eq:lr}
 \ppti \E [f(\theta,\xtheta)] = \E \left[ \ppti f(\theta,\xtheta) + f(\theta,\xtheta)  H_i(\theta,T) \right]
\end{equation}
  where 
  \begin{equation}\label{eq:H}
H_i(\theta,T) =  \sum_{\ell=0}^{N(T)-1} \frac{\ppti \lambda_{k_\ell}(\theta,\hat X_\ell(\theta)) }{\lambda_{k_\ell}(\theta,\hat X_\ell(\theta)) } - \sum_{k=1}^K \int_0^T \ppti \lambda_{k}(\theta,\xts)ds,
\end{equation}
and where 
\begin{itemize}
\item $N(T)$ is the total number of jumps of $\xtheta$ through time $T$, and a sum of the form $\sum_{\ell = 0}^{-1}$ is set to zero,
\item $k_\ell$ is the index of the reaction that changes the system from the $\ell$th state to the $(\ell+1)$st state,
\item $\hat X_\ell(\theta)$ is the $\ell\th$ state in the embedded discrete chain of the path.  
\end{itemize}
For a system \eqref{eq:pmodel} with intensities of the form $\lambda_k(\theta,x) = \theta_k g_k(x)$, where $g_k:\Z^d \to \R_{\ge 0}$, such as stochastic mass action kinetics, $H_i$ simplifies to
\begin{equation*}
H_i(\theta,T) =  \frac{1}{\theta_{i} }\left( N_{i}(T) - \int_0^T \lambda_{i}(\theta,\xts) ds \right) 
\end{equation*}
where $N_{i}(T)$ is the number of jumps of reaction $i$ by time $T$.  
 See  \cite{GlynnAsmussen2007,Glynn1990,PlyasunovArkin}.

The random variable $H_i(\theta,T)$ is often known as a weighting function or weight, and is simple to compute during path simulation.  The likelihood ratio method is  widely applicable, straightforward to use, and provides an unbiased estimate of the sensitivity.  However, the variance of the estimate is often prohibitively large, leading to an inefficient method.  One can reduce this variance significantly by using the weight as a control variate (see e.g. Section V.2 of \cite{GlynnAsmussen2007}), since $H_i(\theta,\cdot)$ is often a mean zero martingale \cite{AndersonKurtzBook}.  
 
\subsubsection{The LR method applied to coupled paths}\label{sec:LR_coupled}
As was pointed out in and around \eqref{eq:2408957},  we want to apply the likelihood ratio method to estimate the sensitivity $\ppti\E [ L_X(\theta) - L_Z (\theta) ]$ where $X$ and $Z$ are coupled processes.  Assume that $\xtheta$ and $\ztheta$ have the same jump directions $\zeta_k\in \Z^d$, but different intensity functions.  Denote their respective intensity functions by $\lambda_k^X$ and $\lambda_k^Z$.  It may happen that $\xtheta$ and $\ztheta$ have different natural state spaces.  In particular, the most common application will have $\xtt \in \Z^d_{\ge 0}$ while $\ztt \in \Z^d$.  Therefore, we simply take the domains of $\lambda_k^X$ and $\lambda_k^Z$ to be the union of the two; for example, all of $\Z^d$.  If the natural domain of either intensity function is some subset of $\Z^d$, then that function will need to be extended to this larger domain in some reasonable fashion. For example, since the natural domain of $\lambda_k^X$ is often $\Z^d_{\ge 0}$, we may extend each $\lambda_k^X$ to be identically zero outside of the non-negative orthant. 
  
To proceed we must couple the process $\xtheta$ and $\ztheta$; i.e. we must build them on the same probability space.  We will use the split coupling, which first appeared in \cite{Kurtz1982} and has since appeared in numerous publications related to computational methods \cite{Anderson2012,AET2014,AH2012,AK2014,GuptaKhammash,WolfAnderson2012}.  We take 
\[
	W_\theta(t) := \left[ \begin{array}{c} \xtt \\ \ztt \end{array}\right]
\]
to be the family of processes satisfying the stochastic equation
\begin{align}\label{eq:coup}
\begin{split}
\xtt &= X_\theta(0) +  \sum_{k=1}^K Y_{k,1}\left(\int_0^t \lambda_{k}^{X}(\theta,\xts)\wedge\lambda_{k}^{Z}(\theta,\zts)ds\right)\zeta_k
\\ & \hspace{1.38cm} + Y_{k,2}\left(\int_0^t \lambda_{k}^{X}(\theta,\xts) - \lambda_{k}^{X}(\theta,\xts)\wedge\lambda_{k}^{Z}(\theta,\zts)ds\right) \zeta_k,\\
\ztt &= Z_\theta(0) +  \sum_{k=1}^K Y_{k,1}\left(\int_0^t \lambda_{k}^{X}(\theta,\xts)\wedge\lambda_{k}^{Z}(\zts)ds\right)\zeta_k
\\ & \hspace{1.38cm} + Y_{k,3}\left(\int_0^t \lambda_{k}^{Z}(\theta,\zts) - \lambda_{k}^{X}(\theta,\xts)\wedge\lambda_{k}^{Z}(\theta,\zts)ds \right) \zeta_k,
\end{split}
\end{align}
where  $\{Y_{k,1},Y_{k,2},Y_{k,3}\}$ are independent unit-rate Poisson processes and we recall that $a\wedge b = \min(a,b)$ for any $a,b\in \R$.
Note that the $2d$-dimensional process $W_\theta(t)$ is also a CTMC.  
For each $k\in \{1,\dots,K\}$ the reaction of the system \eqref{eq:pmodel} with reaction vector $\zeta_k\in\Z^d$ has been associated with three reactions of the process $W_\theta$.  The reaction vectors for these three reactions, which are elements of $\Z^{2d}$, are
\[\eta_{k,1} = \left[ \begin{array}{c} \zeta_k \\ \zeta_k \end{array}\right], \;
\eta_{k,2}= \left[ \begin{array}{c} \zeta_k \\ 0 \end{array}\right], \;
\eta_{k,3}= \left[ \begin{array}{c} 0 \\ \zeta_k \end{array}\right],
\] where each $0$ is interpreted as $\vec 0\in \Z^d$.  Letting $w = \binom{x}{z} \in \Z^{2d}$,
where $x,z\in \Z^d$, the intensity functions for the three reactions are
\begin{align}\begin{split}\label{Lrates}
 \Lambda_{k,1}(\theta,w) &= \lambda_{k}^{X}(\theta,x)\wedge\lambda_{k}^{Z}(\theta,z),\\
\Lambda_{k,2}(\theta,w) &= \lambda_{k}^{X}(\theta,x) - \lambda_{k}^{X}(\theta,x)\wedge\lambda_{k}^{Z}(\theta,z),\\
\Lambda_{k,3}(\theta,w) &= \lambda_{k}^{Z}(\theta,z) - \lambda_{k}^{X}(\theta,x)\wedge\lambda_{k}^{Z}(\theta,z).
\end{split}\end{align}

\noindent We say a reaction associated with $W_\theta$  is of \textit{type} $j\in\{1,2,3\}$ if the reaction vector is $\eta_{k,j}$.
Now note that
\[ W_\theta(t) = W_\theta(0) + \sum_{j=1}^3 \sum_{k=1}^K Y_{k,j} \left(\int_0^t \Lambda_{k,j}(\theta,W_\theta(s)) \, ds\right) \eta_{k,j}
\] 
has the same general form as \eqref{eq:pmodel}.  Thus, as long as the rates satisfy the usual mild regularity conditions, 
we may use the likelihood method as in \eqref{eq:lr}--\eqref{eq:H}.  Given some function
$\tilde f: \R^R\times D_{\Z^{2d}}[0,\infty) \to \R$, the analogous equations are
\begin{equation}\label{eq:coupledlr}
 \ppti \E [\tilde f(\theta,W_{\theta})] = \E \left[ \ppti \tilde f(\theta,W_\theta) + \tilde f(\theta,W_\theta)  \tilde H_i(\theta,T) \right]
\end{equation}
  where 
\[
\tilde H_i(\theta,T) =  \sum_{\ell=0}^{\tilde N(T)-1} \frac{\ppti \Lambda_{k_\ell,j_\ell}(\theta,\hat W_\ell(\theta)) }{\Lambda_{k_\ell,j_\ell}(\theta,\hat W_\ell(\theta)) } - \sum_{j=1}^3\sum_{k=1}^K \int_0^t \ppti \Lambda_{k,j}(\theta,W_{\theta}(s))\, ds , 
\]
and where
\begin{itemize}
\item  $\tilde N(T)$ is the total number of jumps of $W(\theta)$ through time $T$, 
\item $k_\ell\in\{1,\dots,K\}$ is the index and $j_\ell\in\{1,2,3\}$ is the type of the reaction that changes $W_\theta$ from the $\ell$th state to the $(\ell+1$)st state, 
\item $\hat W_\ell(\theta)$ is the $\ell\th$ state in the embedded discrete chain of the path of $W_{\theta}$, with enumeration starting at $\ell = 0$.
\end{itemize}

For a system in which $\Lambda_{i,j}(\theta,w) = \theta_k g_{i,j}(w)$, $\tilde H_i$ simplifies to
\begin{equation*}
\tilde H_i(\theta,T) = \sum_{j=1}^3\left[  \frac{1}{\theta_i}\left( \tilde N_{i,j}(T) - \int_0^T \Lambda_{i,j}(\theta,W_\theta(s))\, ds \right) \right],
\end{equation*}
where $\tilde N_{i,j}(T)$ is the number of jumps of reaction $i$ of type $j$ by time $T$.

We return to our problem at hand of estimating 
\[
	\ppti \E [L_X(\theta)-L_Z(\theta)] = \ppti \E \left[ \int_a^b F(\theta,\xts) - F(\theta,\zts)\; ds \right].
\]
Using \eqref{eq:coupledlr} with $\tilde f (\theta,W_\theta) = \int_a^b [ F(\theta,\xts) - F(\theta,\zts)] ds$, we see that, so long as the differentiation is valid, $\ppti \E[L_X(\theta) - L_Z(\theta)] =\E [V(\theta)]$ with 
\begin{align}\label{lrest}
V(\theta):= \int_a^b \left( \ppti F(\theta,\xts) - \ppti F(\theta,\zts)\right) ds + \tilde H_i(\theta,b)\int_a^b [ F(\theta,\xts) - F(\theta,\zts)] ds,
\end{align}
where the partial of $F$ is always with respect to the first variable.

\subsubsection{Requirements for the process $\ztheta$.}
So long as  the rates of both $\xtheta$ and $\ztheta$ are differentiable, the new rates \eqref{Lrates} for the coupled process are piecewise differentiable.  However, because the intensities $\Lambda_{k,j}$ involve minima, there may be values of $\theta$ and $w$ where the derivative does not exist.  In particular, this may occur if, for some $k$, the two rates in the minimum 
$ \lambda^X_k(\theta,x) \wedge 
\lambda^Z_k(\theta,z)$ are equal, since at such points the left- and right-hand derivatives may be different.

The following condition ensures the differentiability of each $\Lambda_{k,j}$. 
\begin{cond}\label{coupledlrcond}
Suppose for some $k\in\{1,\dots,K\}$ and some $w = \binom{x}{z}$ in the state space of $W$ we have that $\lambda^X_k(\theta,x)=\lambda^Z_k(\theta,z)$.  Then we require that $\ppti\lambda^Z_k(\theta,z)=\ppti\lambda^X_k(\theta,x)$ for each $i\in\{1,\dots,R\}$. 
\end{cond}

\section{The hybrid pathwise method}\label{sec:hyb}

\subsection{ Putting it all together}

Developing  hybrid pathwise methods is now straightforward.  We will estimate $\nabla_\theta \E[L_X(\theta)]$ using \eqref{eq:2408957} for an appropriately chosen process $\ztheta$.  In Section \ref{sec:pwbackground}, we detailed the main conditions that $\ztheta$ must satisfy for this procedure to work.  Specifically, we need a $\ztheta$ that is tightly coupled with $\xtheta$, that satisfies the non-interruptive Condition \ref{nonint}, and that satisfies the regularity Conditions \ref{cond:extra} and \ref{coupledlrcond}.  
We also require that $F$, which determines $L$ via \eqref{eq:409857}, satisfies Condition \ref{Fcond}. Finally, for the validity of the likelihood ratio method on the error term, we require that $\xtheta$ satisfies Condition \ref{lineargrowth}. The hybrid pathwise method then proceeds by 
\begin{enumerate}
\item estimating $\nabla_\theta \E[L_X(\theta) - L_Z(\theta)]$ via Monte Carlo using the LR method as detailed in Section \ref{sec:LR_coupled}, and 
\item estimating $\nabla_\theta \E[L_Z(\theta)]$ via Monte Carlo using the pathwise method as detailed in Section \ref{sec:pwest}.
\end{enumerate}
Denoting by $Q_{X-Z}$ and $Q_{Z}$ the two estimators detailed above, our final estimate for $\nabla_\theta \E[L_X(\theta)]$ is taken to be
\begin{equation}\label{eq:240985999}
	Q_{X}:= Q_{X-Z} + Q_{Z},
\end{equation}
which is trivially unbiased.
We will generate paths independently, in which case
\begin{equation}\label{eq:8776}
	\textup{Var}(Q_X) = \textup{Var}(Q_{X-Z}) + \textup{Var}(Q_Z),
\end{equation}
which can be estimated and used for confidence intervals in the usual way.

Any $\ztheta$ satisfying the above conditions may be used.  In order to make  specific suggestions, we now restrict ourselves to the setting of biochemistry where, as detailed in the introduction, $\zeta_k = \nu_k' - \nu_k$ and the natural state space of $\xtheta$ is $\Z^d_{\ge 0}$.  We will consider two cases:  when $\lambda_k^X$ satisfies stochastic mass action kinetics and when $\lambda_k^X$ satisfies Michaelis--Menten kinetics.

\vspace{.125in}

\noindent \textbf{Stochastic mass action kinetics.}
Suppose that $\lambda_k^X(\theta,x)$ satisfies stochastic mass action kinetics \eqref{eq:mass-action}, in which case  $\lambda^X_k(\theta,x) = \theta_k g_k(x)$.  We define $\lambda_k^X(\theta,x) = 0$ if $x \notin \Z^d_{\ge 0}$.

We now define $\ztheta$ to be the process satisfying \eqref{eq:pmodel} with the following intensity functions.  
For each $k \in \{1,\dots,K\}$ let $\delta_k>0$.  Let $M>0$ be a large number.  Define 
\begin{equation}\label{makz}
 \lambda_k^Z(\theta,z) =  
\begin{cases}
\theta_k\delta_k
 &  \textrm{if}\;\; z_i<\nu_{ki} \; \textrm{for any \textit{i} such that} \; \nu_{ki}>0
\\
\theta_k M &  \textrm{if}\;\; \lambda^X_k(\theta,z) \geq \theta_k M
\\
\lambda^X_k(\theta,z) & \textrm{otherwise}
\end{cases}.
\end{equation}
 Note that in much of $\Z^d_{\ge 0}$ the rates of $Z_\theta$ are identical to those of $\xtheta$.  Note also that $\ztheta$ satisfies the non-interruptive Condition \ref{nonint}, the restrictive regularity Condition \ref{cond:extra}, and the Condition \ref{coupledlrcond} guaranteeing the applicability of the LR method on the coupled processes.  The redefinition of the intensity functions for large values of $\lambda_k^X(\theta,z)$ (by $\theta_kM$) is a consequence of our theoretical results.  If Theorem \ref{thm:pw_first} can be proven with Condition \ref{cond:extra} replaced by Condition \ref{lineargrowth}, as we believe is possible, then the $M$ term could be ignored and we would have
\begin{equation*}
 \lambda_k^Z(\theta,z) =  
\begin{cases}
\theta_k\delta_k
 &  \textrm{if}\;\; z_i<\nu_{ki} \; \textrm{for any \textit{i} such that} \; \nu_{ki}>0
\\
\lambda^X_k(\theta,z) & \textrm{otherwise}
\end{cases}.
\end{equation*}

 \vspace{.125in}
 \noindent \textbf{Michaelis--Menten Kinetics.}  The standard Michaelis--Menten rate is of the form $\lambda^X_k(\theta,x) = \frac{\theta_1 x_k}{\theta_2 + x_k}$  \cite{PetzoldMM}.  Note that near a fixed $\theta$ this rate is  uniformly bounded in $x\ge0$. For some $\delta_k>0$ let
\begin{equation}\label{mmz}
 \lambda_k^Z(\theta,z) =  
\begin{cases}
\frac{\theta_1 \delta_k}{\theta_2 + \delta_k}
 &  \textrm{if}\;\; z_i<\nu_{ki} \; \textrm{for any \textit{i} such that} \; \nu_{ki}>0
\\
\lambda^X_k(\theta,z) & \textrm{otherwise}.
\end{cases}
\end{equation}

\noindent Note that (i)  $\ztheta$ so defined will again have rates that are in agreement with $\xtheta$ for much of $\Z^d_{\ge 0}$, and (ii)  $\ztheta$ satisfies all the conditions outlined above, including the non-interruptive Condition \ref{nonint}.

\vspace{.125in}

It is important to note that the processes $\ztheta$ defined in the manner of  \eqref{makz} or \eqref{mmz} can reach states with negative coordinates, even if the initial condition $Z_\theta(0)$ is in $\Z^d_{\ge 0}$.  This is a consequence of how we overcame the problem that, in general, biochemical processes do not satisfy the non interruptive condition \ref{nonint}.

\subsection{Implementation issues}\label{sec:issues}

In this short section, we make a few points about implementing the hybrid pathwise method.  
\begin{enumerate}

\item \label{implnonint} In the previous section, we were  conservative in redefining \textit{all} intensity functions so that they can never become zero.  However, if a reaction cannot be interrupted by another, then there is no need to redefine the kinetics at zero.  Allowing such intensities to become zero will then improve the performance of the method.  For example, see the model in Section \ref{ex:switch}.

In particular, if the process $\xtheta$ already satisfies the non-interruptive Condition \ref{nonint} and the restrictive Condition \ref{cond:extra}, then the approximate process $Z_\theta$ is unnecessary: one can use pathwise estimates alone to estimate $\ppti\E [L_X(\theta)]$. See Section \ref{ex:bd} for such an example.

\item The best choice for the $\delta_k$ of \eqref{makz} and \eqref{mmz} will be  model-dependent.  
If $\delta_k$ is too large, the process $Z_\theta$ may cease to be a good approximation of $\xtheta$, which will cause the variance of the likelihood ratio estimate of $\ppti \E[L_X(\theta)-L_Z(\theta)]$ to be large. 
On the other hand, taking $\delta_k$  too small makes it very rare that the process $Z_\theta$ makes a jump that the process $\xtheta$ cannot make.  In this latter case, the problem of estimating $\ppti \E[L_X(\theta) - L_Z(\theta)]$ becomes  a problem of estimating a rare event.

In our  numerical experiments, we found   that taking $\delta_k$ to be near one  was a reasonable choice for all the models we considered. Additionally, we have found that $M$ can be taken arbitrarily large with no loss of accuracy.

\item
If the sensitivity we wish to estimate is of the form $\ppti \E [f(X_\theta(T))]$, i.e.~is not an integral of a function, we may instead write
\begin{align}\label{relatedeq}
\ppti\E [f(X_\theta(T))] &= \ppti\E [ f(X_\theta(T))-f(Z_\theta(T)) ] + \ppti \E [f(Z_\theta(T))],
\end{align}
and note that the LR method is applicable on the first term on the right-hand side of the above equation.  That is, there is no reason to replace $f$ in that term with an integrated function.  The final term must be estimated using either the  GS smoothing method or the RPD smoothing method.
We shall refer to these hybrid procedures for estimating $\ppti\E [f(X_\theta(T))]$ as the \textbf{GS hybrid} and the \textbf{RPD hybrid} methods, respectively.

\item \label{opt}  One must decide how many simulated paths will be used for each of the estimators $Q_{X-Z}$ and $Q_Z$ of 
\eqref{eq:240985999}.  Suppose one wishes to minimize the expected time required to compute an estimate such that the $95\%$ halfwidth is within some target value, $\epsilon$.  That is, we would like 
\begin{equation}\label{eq:constraint}
	\textup{Var}(Q_{X-Z}) + \textup{Var}(Q_Z) = \textup{Var}(Q_X)  \le \delta := \left(\frac{\epsilon}{1.96}\right)^2,
\end{equation}
where $\delta$ denotes the target variance.
Let $v_\ell$ denote the variance $\textup{Var}(V(\theta))$, where $V(\theta)$ is as in \eqref{lrest}, so that $\textup{Var}(Q_{X-Z}) = \frac{v_\ell}{n_\ell}$, where $n_\ell$ is the number of coupled paths simulated.  Also let $c_\ell$ denote the average time required to compute one pair of coupled paths for the likelihood estimate.  Similarly define $v_p$, $c_p$, and $n_p$ for the pathwise estimates.  Then, we wish to minimize the expected total computational time 
$$ 
	n_\ell c_\ell + n_p  c_p = \frac{v_\ell c_\ell}{\textup{Var}(Q_{X-Z})} + \frac{v_p c_p}{\textup{Var}(Q_{Z})}
$$
subject to the constraint \eqref{eq:constraint}.
The solution to this optimization problem satisfies  
\begin{equation}\label{eq:optvar} 
	 \textup{Var}(Q_{X-Z}) = \frac{\delta\sqrt{v_\ell c_\ell}}{\sqrt{v_pc_p}+\sqrt{v_\ell c_\ell}} \quad \text{and} \quad	\textup{Var}(Q_{Z}) = \frac{\delta\sqrt{v_pc_p}}{\sqrt{v_pc_p}+\sqrt{v_\ell c_\ell}}.
\end{equation}
In practice, one may use the following optimization procedure.  First, in a preliminary simulation compute $n$ samples each of   $(X_\theta,Z_\theta)$ and $Z_\theta$.  
Second, from these preliminary samples estimate each of  $v_p, c_p, v_\ell,$ and $c_\ell$ and utilize these values to estimate the target variances \eqref{eq:optvar}.   

\item Finally, we point out that if one first simulates many paths of $Z_\theta$ for use in the pathwise estimate $Q_Z$ and notes that each path is a valid realization of the original process $\xtheta$ (which is simple to check as simulation occurs), then with high probability one knows without further computation that $Q_{X-Z}$ is zero or near zero.  Of course, theoretical work is needed to quantify what is meant by ``high probability'' in the previous sentence.  However, this observation provides a means to check for practical applicability of  pathwise methods, which have been shown to be extremely efficient on many models \cite{Khammash2012}.

\end{enumerate}

\section{Numerical Examples}\label{sec:hybex}

With the examples in this section, we demonstrate the validity and efficiency of our new class of  methods.    An important example is given in Section \ref{ex:switch}, where we demonstrate that pathwise-only methods of the type developed in  \cite{Khammash2012} can fail, in the sense that there can be large biases, if  interruptions can occur.  That is, the example in Section \ref{ex:switch} shows that the error term utilized in this paper, and differentiated using the LR method,  is necessary.

On a variety of examples we compare the efficiency of  the developed methods with  the following:
\begin{enumerate}
\item The likelihood ratio method including the weight \eqref{eq:H} as a control variate (LR+CV).
\item The regularized pathwise derivative method (RPD).
\item The coupled finite difference method (CFD)  using  centered differences.
\item The Poisson path approximation method (PPA).
\end{enumerate}
We will demonstrate that the new methods introduced here compare quite favorably with this group of already established methods, with the GS hybrid method often the most efficient unbiased method.
Future work will involve a wider numerical study to help determine a better framework for choosing the most efficient method for a given model.

Throughout, we use the term ``variance'' to refer to estimator variance, which is the sample variance divided by the number of paths simulated.   For each hybrid method estimate, we use the optimization procedure described in item \ref{opt} of Section \ref{sec:issues}, and compute the variance as in \eqref{eq:8776}. All half-widths given are $95\%$ confidence intervals computed as $1.96$ multiplied by the square root of the variance.  The numerical calculations were carried out in MATLAB using an Intel i5-4570 3.2 GHz quad-core processor.

\subsection{Birth-death}\label{ex:bd}

Consider the birth-death model 
\[ \emptyset \underset{\theta_2}{\overset{\theta_1}{\rightleftarrows}} A\]
with mass action kinetics.  We let  $\xtt$ denote the abundance of  $A$ at time $t$ and take $X_\theta(0)=0$.  For this model, we can solve to find that 
\[\E[ \xtt ]= \frac{\theta_1}{\theta_2} ( 1 - e^{-\theta_2 t} ), \]
and
\[ \frac{\partial}{\partial\theta_1} \E[ \xtt]
= \frac{1}{\theta_2} ( 1 - e^{-\theta_2 t} )
\quad \textrm{and} \quad
\frac{\partial}{\partial\theta_2} \E [\xtt]
= \frac{\theta_1}{\theta_2} ( t e^{-\theta_2 t} ) 
- \frac{\theta_1}{\theta_2^2} ( 1 - e^{-\theta_2 t} ).
\]
We estimate the sensitivity with respect to $\theta_2$  of the quantity $\E [\xtt]$ at  $\theta_0=(\theta_1,\theta_2)=(10,0.5)$. 

Since the model naturally satisfies Condition \ref{nonint} we may use the GS Pathwise and RPD methods without the error terms; see item \ref{implnonint} of Section \ref{sec:issues}.  Though the intensity of the model is unbounded, the intensities are ``bounded in practice:'' throughout these simulations no intensity was ever greater than $M=10^3$.  That is, if we had used the full hybrid method with an approximate process $Z_\theta$ with an intensity bounded above by $10^3,$ then the error term would have given us an estimate of zero.  We may therefore confidently use both pathwise-only methods. 

 Figure \ref{fig:bdestvar} shows that each method does a good job of estimating the given sensitivities and that the GS pathwise method has the lowest variance of any unbiased method. In fact, for this experiment the GS pathwise method also has a smaller variance than most of the biased methods.
The RPD method, with the larger choice of $w$, has a slightly lower variance than the GS pathwise method.

\begin{figure}
\centering
\textbf{Comparison of $\theta_2$ sensitivity estimates, Birth-Death model}
\vskip1ex
$\begin{array}{ll}
\includegraphics[width=3in]{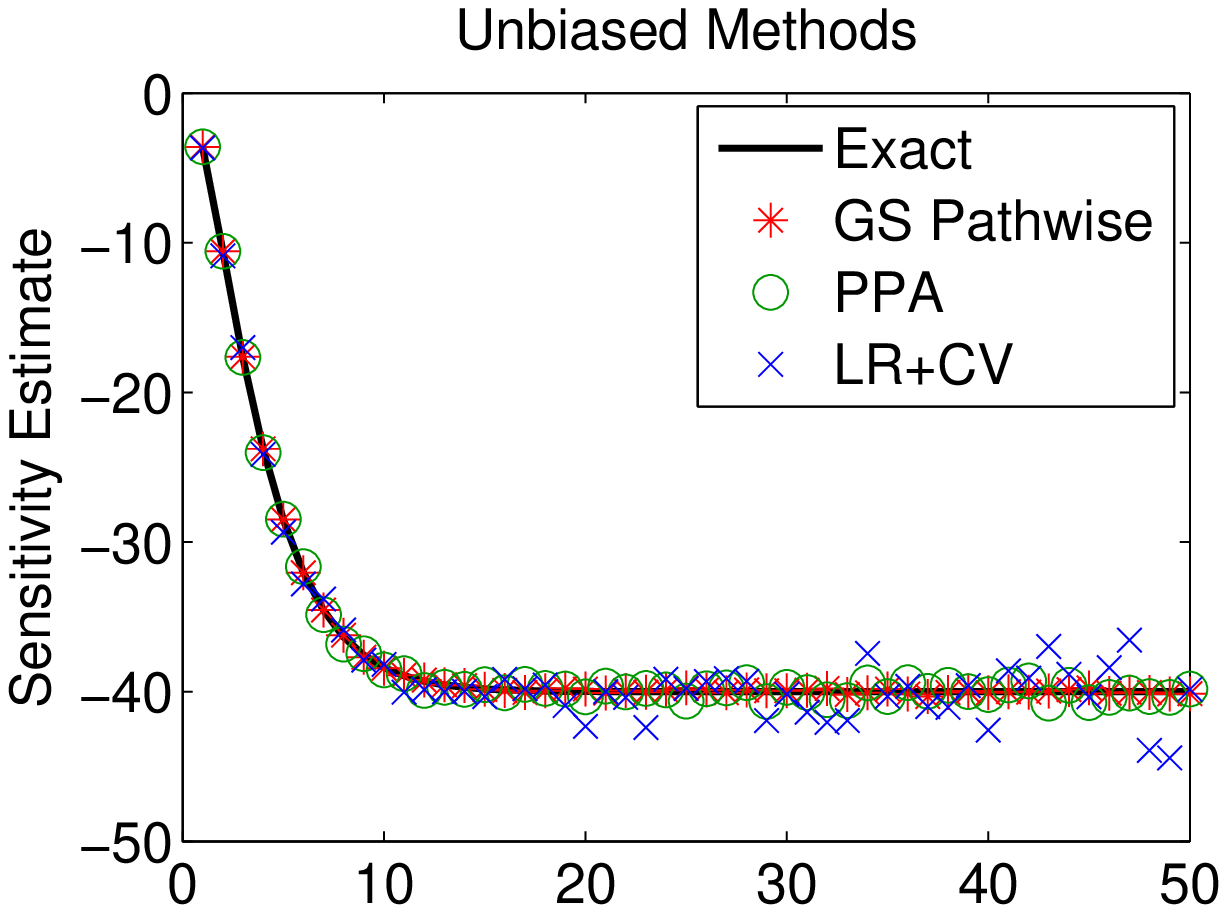} &  \includegraphics[width=3in]{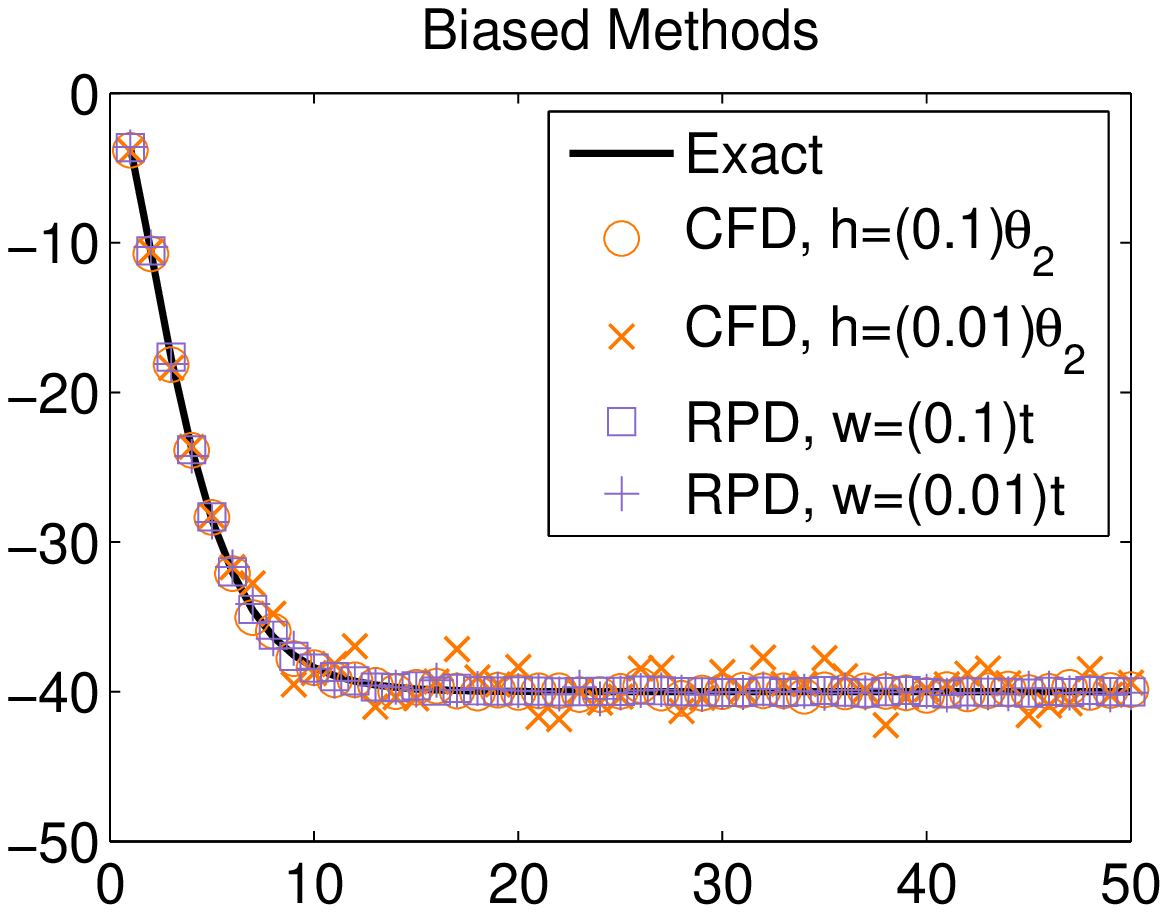} 
\\
\includegraphics[width=3in]{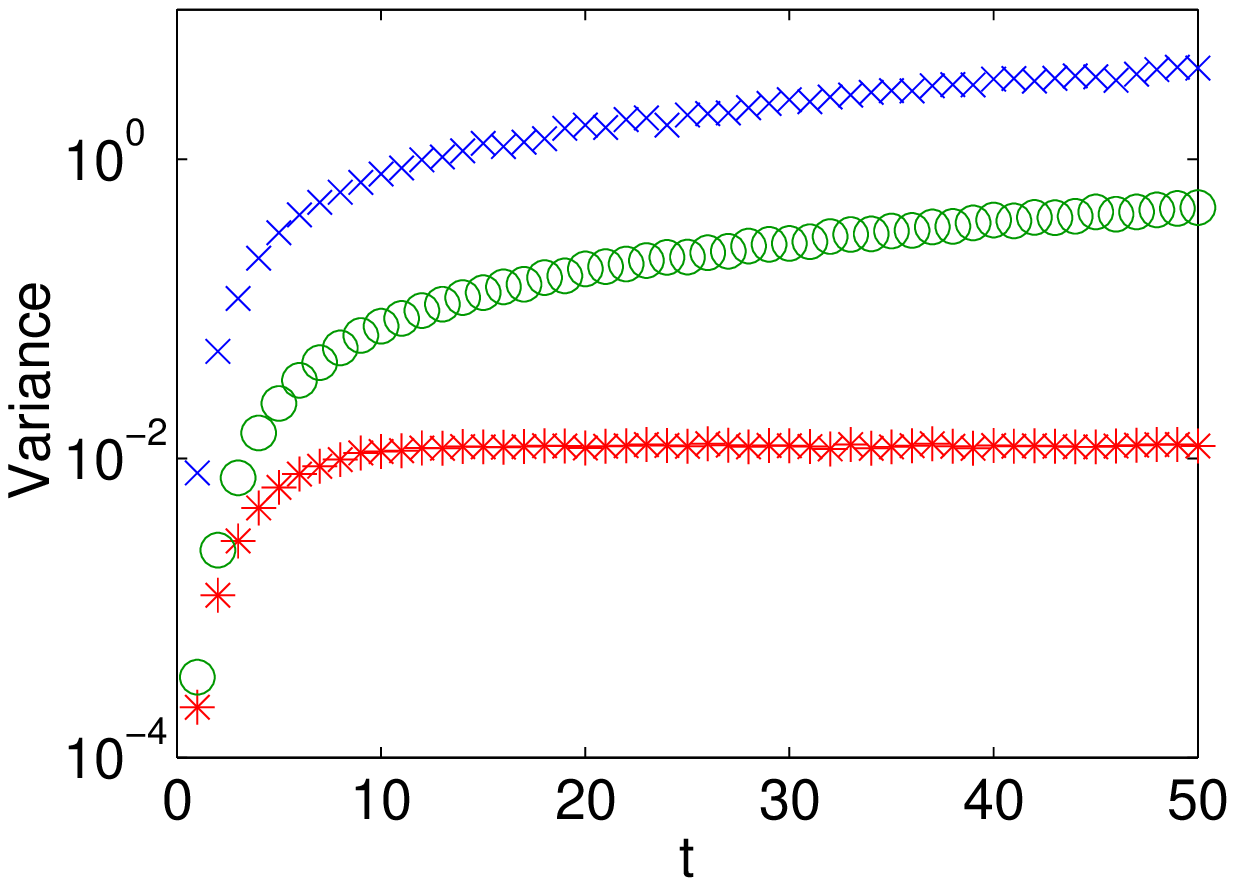} &  \includegraphics[width=3in]{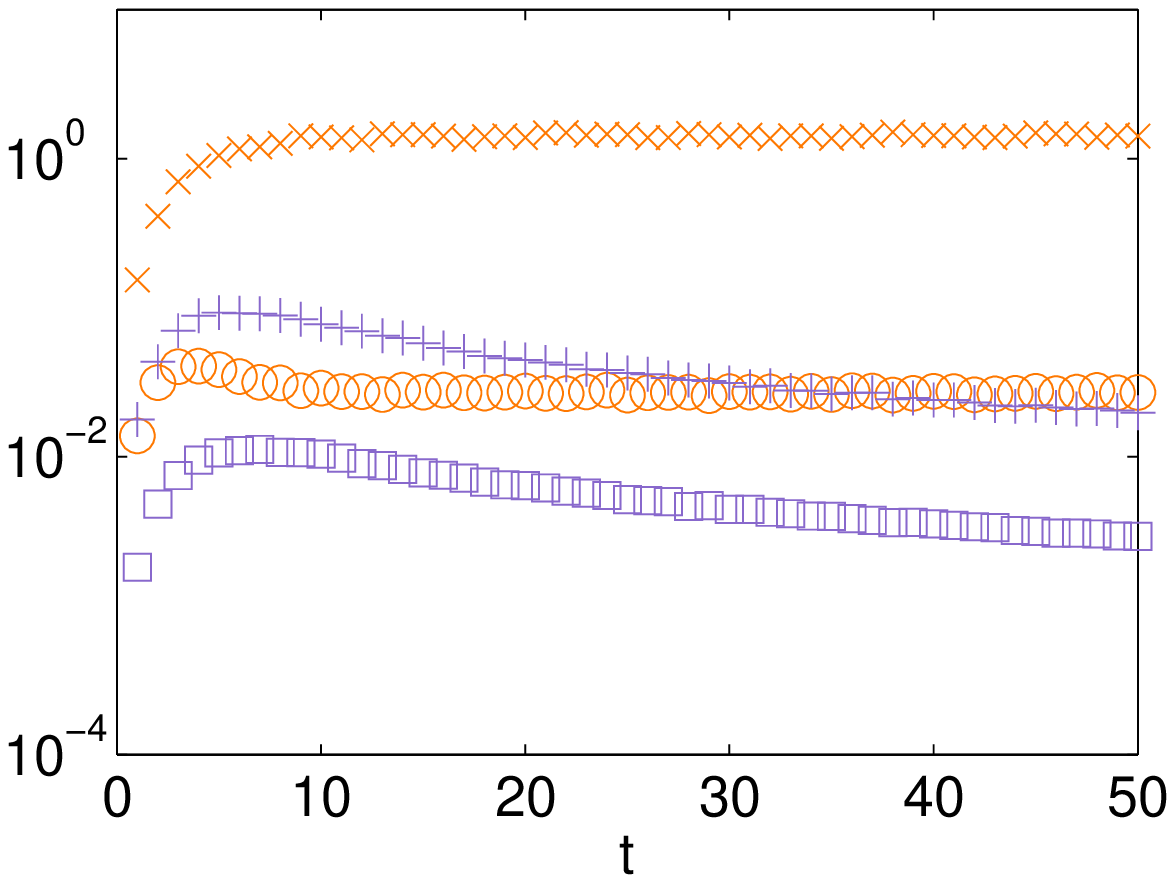} 
\end{array}$
\caption{A comparison of the sensitivity estimates and method variance for the birth-death model of Section \ref{ex:bd} as the time $t$ is varied. $10^4$ paths were used for  each method, and  $X_\theta(0)=0$, and $\theta_0=(10,0.5)$.   The parameter $h$ for the CFD method was chosen as a fraction of the parameter $\theta_2$.  Similarly, the parameter $w$ for the RPD method was chosen as a fraction of the time $t$, which varies in this experiment.
}
\label{fig:bdestvar}
\end{figure}

\begin{figure}
\centering
\textbf{Comparison of efficiency for $\theta_2$ sensitivity estimation, Birth-Death model}
\vskip1ex
\includegraphics[width=3in]{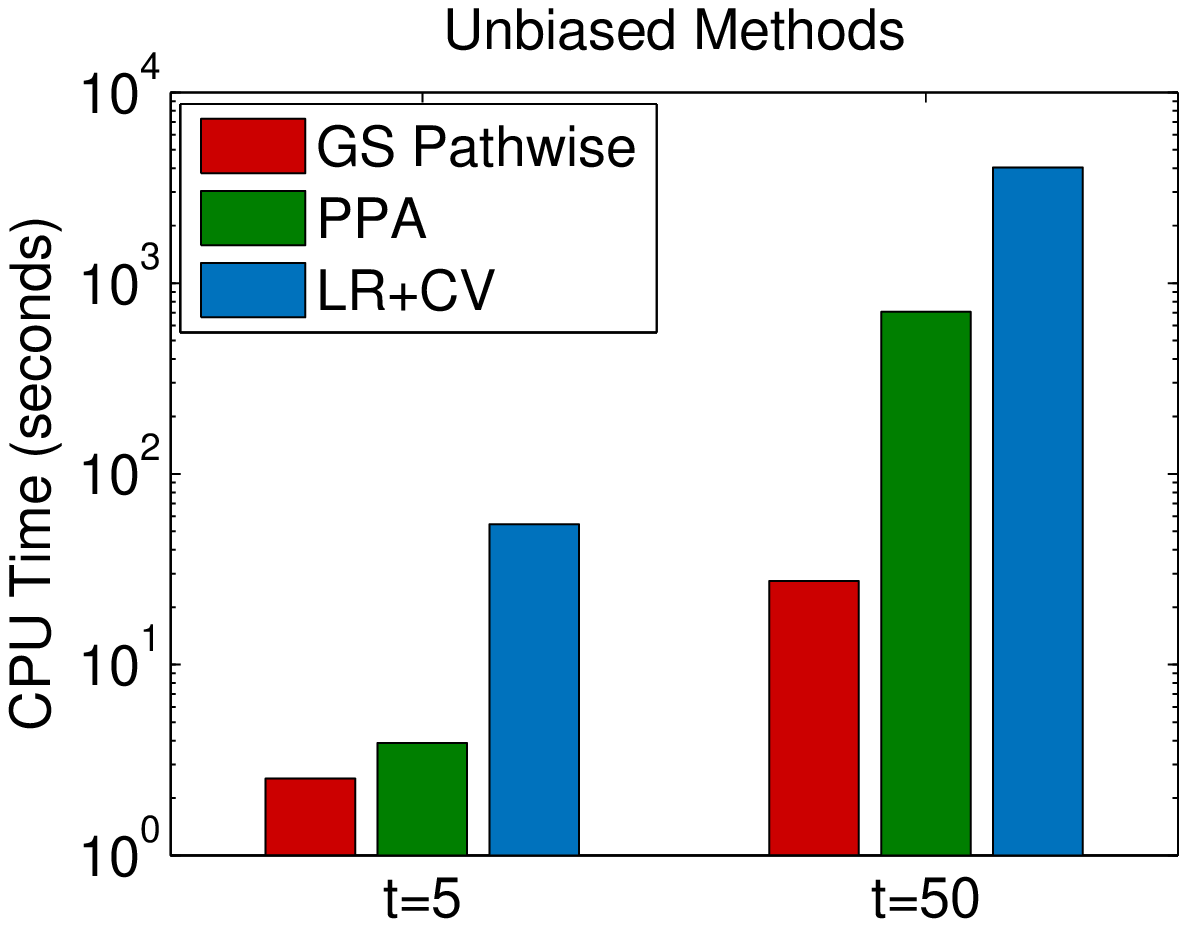}   \includegraphics[width=3in]{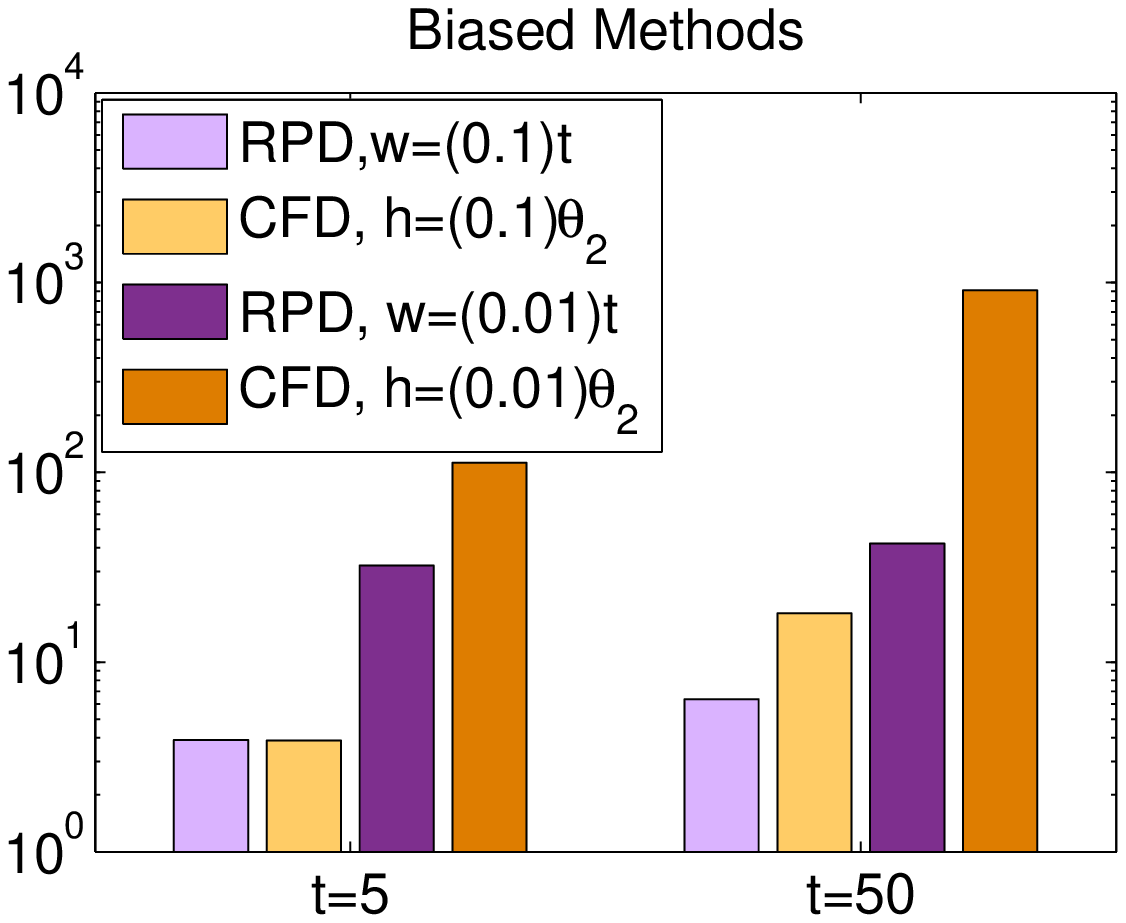}
\caption{A comparison of the efficiency of the different methods in estimating the sensitivity with respect to $\theta_2$ of the birth-death model of Section \ref{ex:bd} with $X_\theta(0)=0$, and $\theta_0=(10,0.5)$. Two different times, 5 and 50, were used. The CPU times reported are the times required by the different  methods to produce a target confidence interval of half-width equal to 1\% of the absolute value of the sensitivity.  Note that 
a log scale is used.
}
\label{fig:bdbar}
\end{figure}

A more straightforward comparison of method efficiency can be provided by finding the CPU time required for each method to estimate the sensitivity to  a given tolerance.  In  Figure \ref{fig:bdbar}, we report these CPU times when we run each method until it produced a half-width equal to 1\% of the absolute value of the sensitivity.  As can be seen in the figure, the GS pathwise method is significantly more efficient than the other unbiased methods.  Indeed, at time $t=5$, the GS pathwise method is over 3 times faster than PPA, and more than 20 times faster than the LR+CV method.  At time $t=50$, the GS pathwise method is over 25 times faster than PPA, and nearly 150 times faster than the LR+CV method.   At time $t=5$, the GS pathwise method is also more efficient than any of the biased methods used.  

The efficiency of the biased methods RPD and CFD is highly influenced by the choice of the parameter $w$ or $h$.  At time $t=50$, the RPD method with $w=(0.1)t=5$ is over 4 times faster than the GS pathwise method, though at the cost of a small bias.

Finally, we note here that on this model and the other models simulated, the LR+CV method, which uses the weighting function as a control variate, generally has variances at least an order of magnitude smaller than the usual LR method in which a control variate is not used.  The additional computational cost of adding this control variate  is negligible.  

\subsection{A simple switch}\label{ex:switch}
In contrast to the linear growth model, the following simple switch is one in which the two pathwise-only methods can have a large bias if no correction term is added:
\[
	A \overset{\theta_1}{\rightarrow} \emptyset, \quad A \overset{\theta_2}\rightarrow B, \quad B \overset{\theta_3}{\rightarrow} C,
\]
with $X_\theta(0)=(a,0,0)$ giving  the initial abundances of $A, B,$ and $C$ respectively.  We estimate the derivative with respect to $\theta_1$ of the mean number of $C$ molecules, $\frac{\partial}{\partial\theta_1}\E [X_{\theta,C}(t)]$, at $\theta=(\frac{1}{4},1,1)$ and at various times $t$.
Since this model is linear, we can solve for the sensitivity exactly at $\theta=(\theta_1,1,1)$:
\[\frac{\partial}{\partial\theta_1}\E [X_{\theta,C}(t) ] = \frac{ae^{-t}}{\theta_1^2} - \frac{a}{(1+\theta_1)^2} - ae^{-(1+\theta_1)t}\bigg(\frac{\theta_1^2 t +\theta(t+2)+1 } { \theta_1^2(1+\theta_1)^2}
\bigg).
\]

\subsubsection{Pathwise-only methods are biased}

We consider the bias of the GS pathwise and RPD methods in computing the sensitivity $\frac{\partial}{\partial\theta_1} \E [X_{\theta,C}(t)]$.
For the GS pathwise method we  use  $\E[X_{\theta,C}(t)] = \E[ \int_0^t X_{\theta,B}(s)\, ds],$ which follows from \eqref{eq:dynkin}. 
For the RPD method, we use
\[  
	 \E\left[ \frac{1}{2w} \int_{T-w}^{T+w} X_{\theta,C}(s)\, ds\right]
\]
as an approximation to $\E[X_{\theta,C}(T)]$.
As shown in Figure \ref{fig:switchrpd}, the RPD and GS pathwise methods  provide  biased estimates, with the bias ranging from small to (very) large, depending on the initial condition and time, $t$. 
\begin{figure}
\centering

\textbf{Error of pathwise-only methods, switch model}

\includegraphics[width=3in]{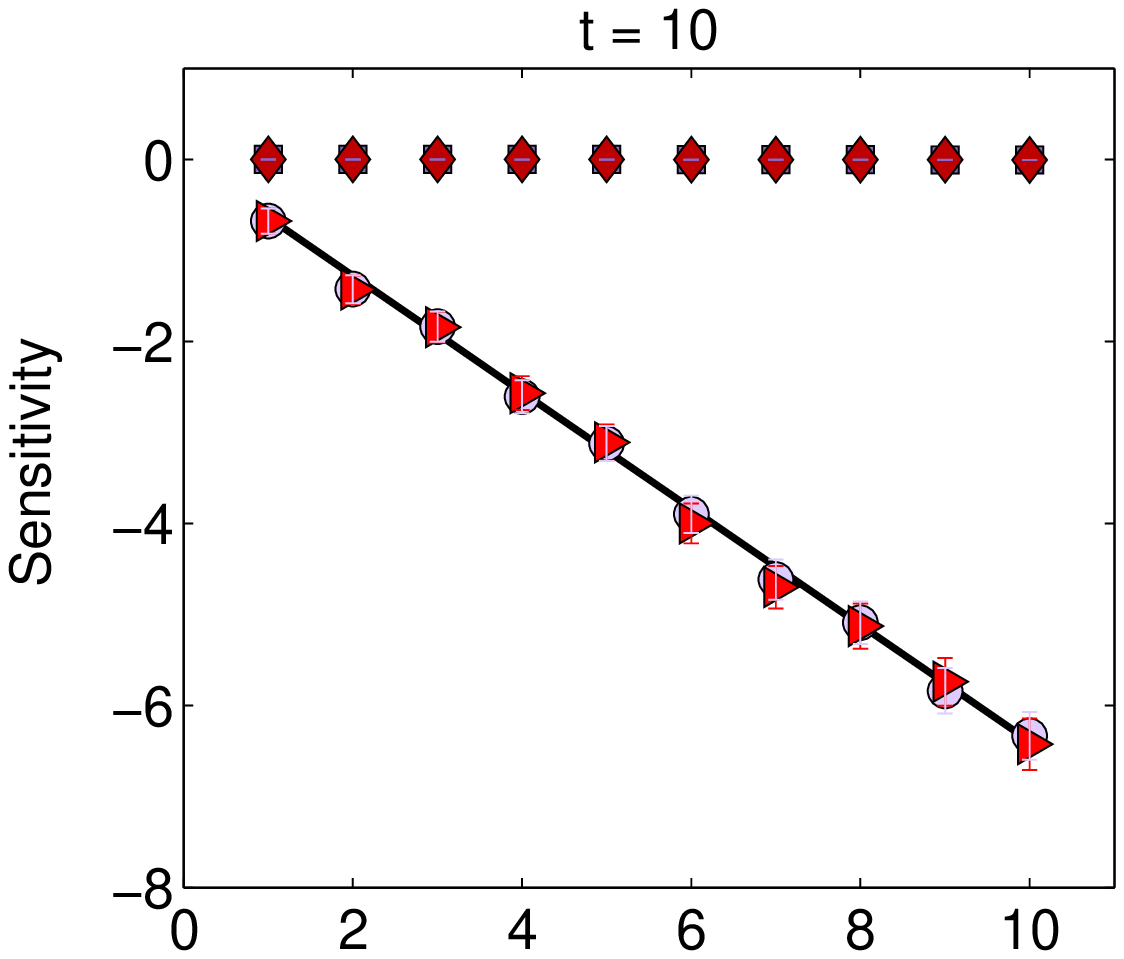}\\
\includegraphics[width=3in]{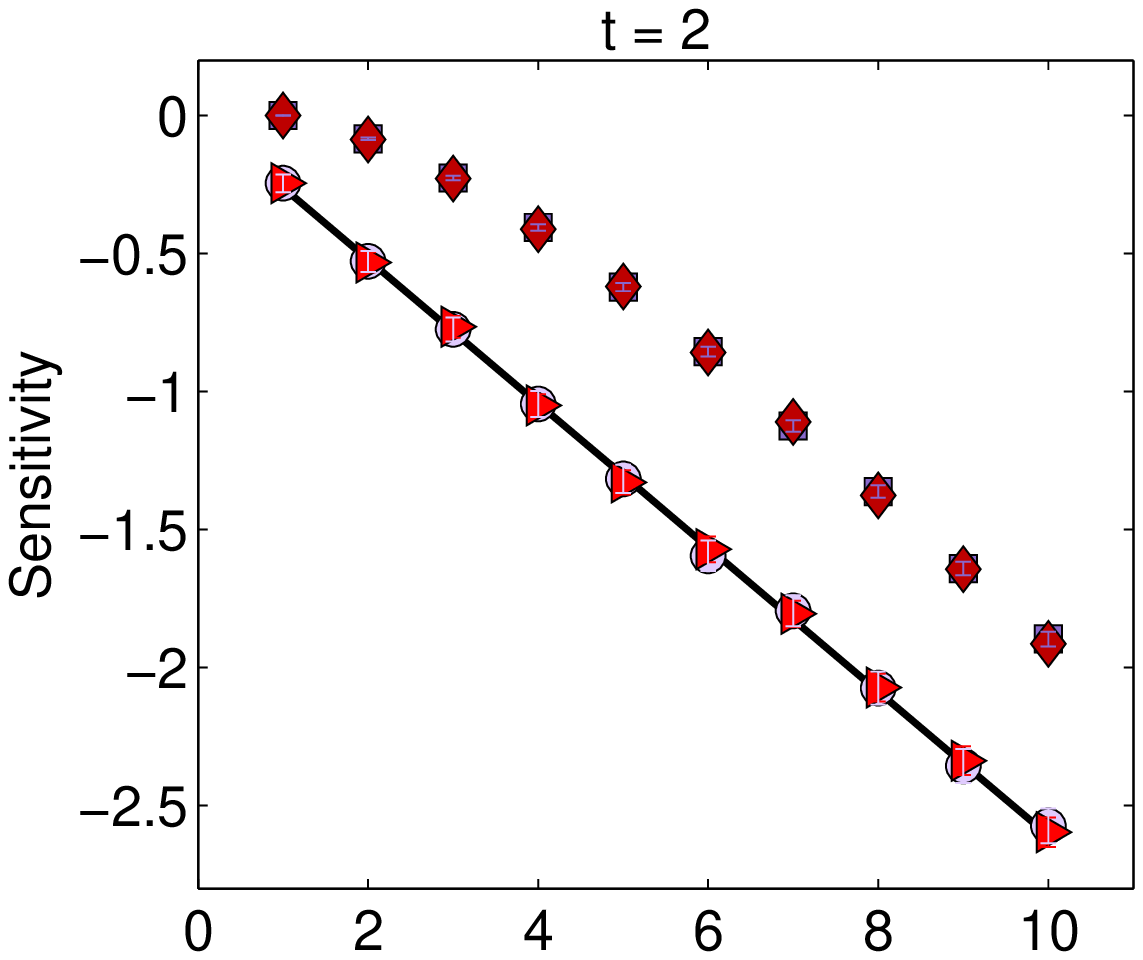}\\
\includegraphics[width=3in]{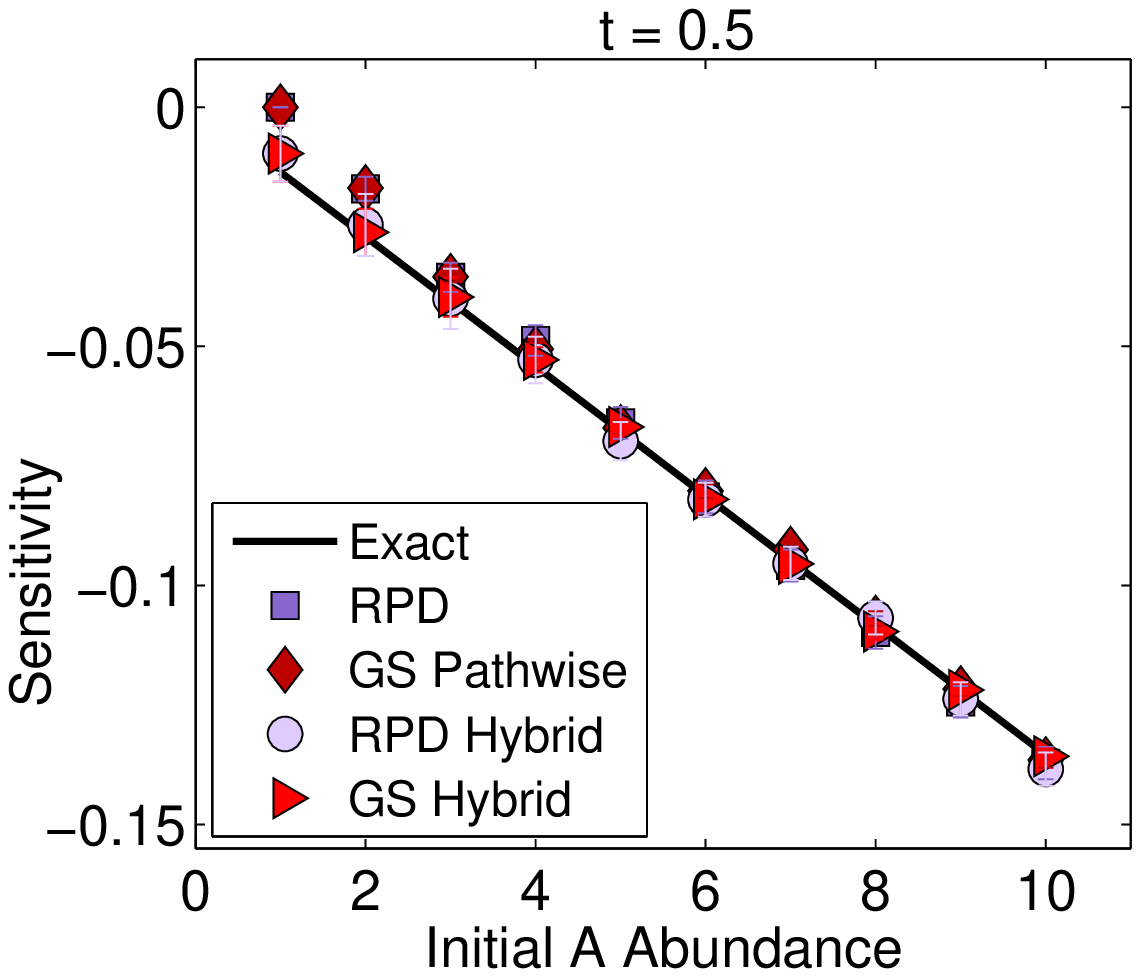}
\caption[Error of pathwise-only methods, Switch]{A demonstration of the significant bias of the pathwise-only methods (RPD and GS) for the estimation of the sensitivity of $\E[ X_{\theta,C}(t)]$ with respect to $\theta_1$ in the switch model of Section \ref{ex:switch}.  Various initial $A$ abundances and three different times $t$ are used. The GS Hybrid and RPD Hybrid method estimates are also shown; both  estimate the exact sensitivity well. Each estimate used $10^5$ paths, and a value of $w=(0.1)t$ was used for both the RPD and RPD Hybrid methods. For $t=10$, both hybrid methods used $30\%$ pathwise estimates (and $70\%$ coupled likelihood estimates); at $t=2$, both used $75\%$ pathwise estimates; and at $t=0.5$, both used $90\%$ pathwise estimates.}
\label{fig:switchrpd}
\end{figure}
 In fact, at $t=10$, these two methods provide estimates of approximately zero for a sensitivity of magnitude approximately 6.
At a small time of $t=0.5$, the  RPD and GS pathwise methods show only a small bias, though it is still noticeable for small initial abundances of $A$.    
In each plot of Figure \ref{fig:switchrpd}, the same value of $w$ was used for both the RPD and RPD Hybrid methods (the hybrid methods for this example are discussed below).  

These results confirm that neither the RPD method nor the GS pathwise method is unbiased for models with interruptions.  Further, the biases can be substantial.  

\subsubsection{Comparison of valid methods}\label{sec:switchrates}
 To use the hybrid methods introduced in this paper, we 
 construct $Z_\theta$ as in Section \ref{sec:hyb} with

\begin{equation}\label{eq:switchrates} 
\lambda_1^Z(\theta,z) =  
\begin{cases}
\frac{1}{4}
 &  z_A < 1
\\
\frac{1}{4}z_A & \textrm{otherwise}
\end{cases}\;\;,\quad  \lambda_2^Z(\theta,z) =  
\begin{cases}
1
 &  z_A < 1
\\
 z_A  & \textrm{otherwise}
\end{cases}\;\;, \quad \lambda_3^Z(\theta,z) = \begin{cases}
0 & z_B < 1\\
z_B & \textrm{otherwise}
\end{cases}.
\end{equation}

\noindent The process $Z_\theta$ may now reach states in which the first coordinate  is negative.
We may allow the rate $\lambda_3^Z(\theta,x)$ to be zero because the reaction $B \to C$ can never be interrupted by another reaction; see item \ref{implnonint} of Section \ref{sec:issues}. 
Hence, the $Z_\theta$ constructed with rates \eqref{eq:switchrates} is still non-interruptive.

In Figure \ref{fig:switchbar}, we give a comparison of method efficiency with $a=10$ for times $t=0.5, t=2$ and $t=10$. Again, we give the time required for each method to achieve a confidence interval of half-width equal to $1\%$ of the magnitude of the sensitivity.   At $t=0.5$, the GS Hybrid method is significantly more efficient than any other method; in particular, it is almost 10 times faster than PPA and over 165 times faster than LR+CV, the other unbiased methods considered.  As time increases to $t=2$, however, PPA becomes the most efficient method.  At $t=10$ the advantage of PPA over the hybrid methods is even more significant: PPA is over 30 times faster than the GS Hybrid method.
Interestingly, at $t=10$, the LR+CV method is very nearly as efficient as PPA.  
This is a particularly striking example of why future work should include a study of the regimes in which a given method is likely to be the most efficient choice.

Note that in this example the biased methods with the given parameter choices are less efficient than the most efficient unbiased method at each time we considered.
   
\begin{figure}
\centering
\textbf{Comparison of efficiency for $\theta_1$ sensitivity estimation, switch model}
\includegraphics[width=6in]{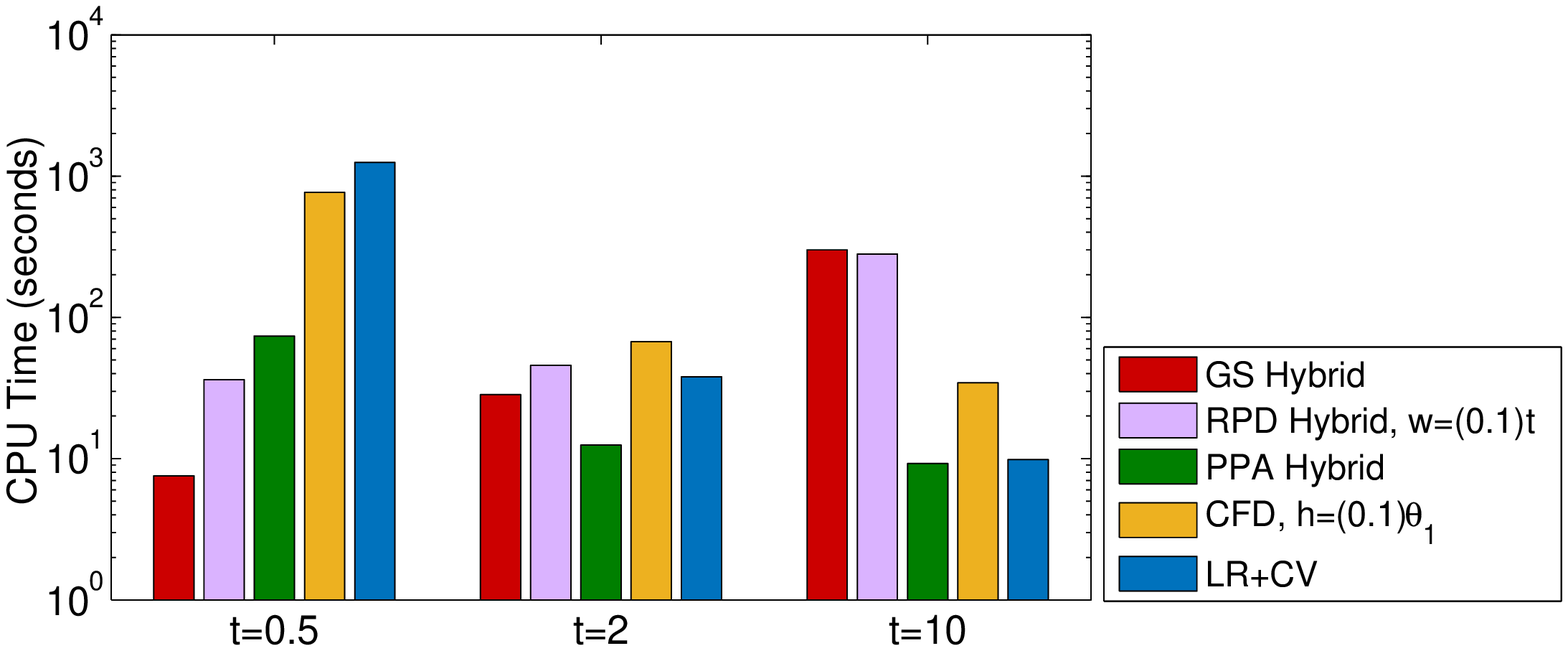}
\caption[1st order method comparison, Switch]{An efficiency comparison for the estimation of the sensitivity of $\E[ X_{\theta,C}(t)]$ with respect to $\theta_1$ in the switch model of Section \ref{ex:switch}, with $a=10$.  CPU time gives the computation time in seconds required to achieve a confidence half-width of $1\%$ of the sensitivity.  
Via the optimization procedure described in Section \ref{sec:issues}, the GS Hybrid method used approximately $36\%$ pathwise estimates, versus $64\%$ coupled likelihood ratio estimates, when $t=10$; when $t=2$, the method used $76\%$ pathwise estimates, and when $t=0.5$  it used $100\%$ pathwise estimates.  That is, the best allocation strategy is significantly different at these various times.  The RPD Hybrid method similarly uses more pathwise estimates at smaller times, though the exact allocation is different for the two choices of the parameter $w$.  For both hybrid methods, the optimization step is included in the computation time.  The time required for the optimization step, which for this experiment included sampling 500 pathwise estimates and 500 coupled likelihood estimates, was approximately $0.10$ seconds for $t=0.5$, $0.15$ seconds for $t=2$, and $0.25$ seconds when $t=10$.}
\label{fig:switchbar}
\end{figure}

\subsubsection{Michaelis--Menten kinetics} \label{ex:MM}
We  demonstrate the hybrid methods on a non-mass action model.  In particular, the standard Michaelis--Menten approximation of the substrate--enzyme model
\[S \rightarrow \emptyset, \quad  E + S \rightleftarrows ES \rightarrow E + P, \quad P \rightarrow \tilde P 
\]
would lead to the model
\[
	S \overset{\theta_1}{\rightarrow} \emptyset, \quad S \overset{*}\rightarrow P, \quad P \overset{\theta_3}{\rightarrow} \tilde P,
\]
where the intensity $(*)$ is given by $\lambda_2^X(\theta,X_\theta)= \frac{\theta_2 X_{\theta,S}}{\theta_4 + X_{\theta,S}}$, and where $X_{\theta,S}$ denotes the number of substrate molecules. The other two rates follow mass action kinetics.  See for example \cite{PetzoldMM}, from which we obtained the relevant parameter values, $\theta = (1/20, 1, 1, 11)$.  Note that this network is analogous to the switch model above.  For the needed approximate model we use
\[ 
\lambda_1^Z(\theta,z) =  
\begin{cases}
\frac{1}{20}
 &  z_S < 1
\\
\frac{1}{20}z_S & \textrm{otherwise}
\end{cases}\;\;,
\quad \lambda_2^Z(\theta,z) =  
\begin{cases}
\frac{\theta_2 }{\theta_4 +1}
 &  z_S < 1
\\
 \frac{\theta_2 z_S}{\theta_4 + z_S}  & \textrm{otherwise}
\end{cases}\;\;,\quad  \text{and} \quad\lambda_3^Z(\theta,z) = \begin{cases}
0 & z_P < 1\\
\theta_3 z_P & \textrm{otherwise}
\end{cases}.
\]
Again note that the third reaction cannot be interrupted.  We estimate $\frac{\partial}{\partial\theta_1} \E [X_{\theta,\tilde P}(t)]$ at times $t=2$ and $t=20$; the actual sensitivity values are approximately $0.23$ and $29$ respectively.
The results are similar to the results of the mass action switch model of Section \ref{sec:switchrates}. See Figure \ref{fig:mmbar}.  In particular, for the small time $t=2$, the hybrid methods are more efficient than PPA and the other methods.  In particular, the GS Hybrid method is over 7 times faster than PPA.  At the  time of $t=20$, when the intensity of each reaction channel in the system is often zero, the PPA and LR+CV methods are most efficient, with PPA returning the desired estimate over 12 times faster than the GS Hybrid method.

\begin{figure}
\centering
\textbf{Comparison of efficiency for $\theta_1$ sensitivity estimation, Michaelis-Menten switch model}
\includegraphics[width=5in]{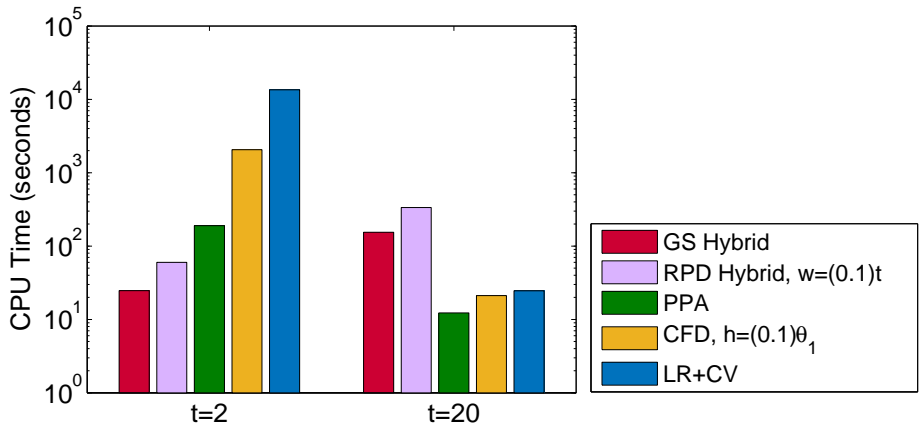}
\caption[1st order method comparison, Michaelis--Menten]{An efficiency comparison for the estimation of $\frac{\partial}{\partial\theta_1} \E [X_{\theta,\tilde P}(t)]$ in the Michaelis--Menten switch model of Section \ref{ex:MM} with an initial $S$ quantity of $10$.  CPU time gives computation time in seconds  required to achieve a confidence half-width of $1\%$ of the sensitivity value.  
}
\label{fig:mmbar}
\end{figure}

\subsection{Dimerization}
\label{sec:ex:dimer}

We  consider a model of mRNA transcription and translation in which, additionally, the protein dimerizes.  
 Table \ref{table:dimerrates} gives the reactions of the model.  Since the model does not satisfy the non-interruptive Condition \ref{nonint}, this table also provides the rates that were used for the approximate process $Z_\theta$ in the hybrid methods.

\begin{table}[h]
\centering
\begin{tabular}{lll|ll}
 & Reaction& $\lambda_k^X$ & \hspace{0.5in}$\lambda_k^Z$\\
\hline
1.) transcription & $\emptyset \hspace{.1cm}\rightarrow\hspace{.1cm} M$ & $\theta_1$ &  $\;\;\;\theta_1$
\\&&&\\
2.) translation & $M{\rightarrow} \hspace{.15cm}M + P$ & $\theta_2X_M$ & $\begin{cases}
\theta_2  & \hspace*{1.5cm} Z_M < 1
\\
\theta_2 \tilde M &  \hspace*{1.5cm}  \theta_2Z_M \geq \theta_2 \tilde M
\\
\theta_2Z_M & \hspace*{1.5cm} \textrm{otherwise}
\end{cases}$
\\
&&&\\
3.) dimerization & $P + P \hspace{.15cm}{{\rightarrow}} \hspace{.15cm} D$ & $\theta_3X_P(X_P-1)$ & 
$\begin{cases}
\theta_3  &  Z_P < 2
\\
\theta_3 \tilde M &  Z_P \geq 2  \quad \textrm{and} \\ & \theta_3Z_P(Z_P-1) \geq \theta_3 \tilde M
\\
\theta_3Z_P(Z_P-1) & \textrm{otherwise}
\end{cases}$
\\
&&&\\
4.) degradation & $M\hspace{.15cm}{{\rightarrow}} \hspace{.15cm} \emptyset$& $\theta_4X_M$ & 
$\begin{cases}
\theta_4 \tilde M &  \hspace*{1.5cm}  \theta_4Z_M \geq \theta_4 \tilde M
\\
\theta_4Z_M & \hspace*{1.5cm} \textrm{otherwise}
\end{cases}$
\\
&&&\\
5.) degradation & $P\hspace{.15cm}{{\rightarrow}} \hspace{.15cm} \emptyset$& $\theta_5X_P$ & 
$\begin{cases}
\theta_5  & \hspace*{1.5cm} Z_P < 1
\\
\theta_5 \tilde M &  \hspace*{1.5cm}  \theta_5Z_P \geq \theta_5 \tilde M
\\
\theta_5Z_P & \hspace*{1.5cm} \textrm{otherwise}
\end{cases}$
\\
&&&\\
6.) degradation & $D\hspace{.15cm}{{\rightarrow}} \hspace{.15cm} \emptyset$& $\theta_6X_P$ &  $\begin{cases}
\theta_6 \tilde M &  \hspace*{1.5cm}  \theta_6Z_D \geq \theta_6 \tilde M
\\
\theta_6Z_D & \hspace*{1.5cm} \textrm{otherwise}
\end{cases}$
\\
\end{tabular}
\caption{Reactions and hybrid rates for the dimerization model of Section \ref{sec:ex:dimer}.
  We take all initial quantities equal to zero and $\tilde M=10^6$ (we have added a tilde to avoid confusion with the symbol for mRNA). For the process $Z_\theta$ to be non-interruptive, we need only prevent three of the intensities from being zero: $\lambda_2, \lambda_3,$ and $\lambda_5$. Indeed, $\lambda_1$ is constant, and reactions 4 and 6 cannot be interrupted by another reaction.  
}
\label{table:dimerrates}
\end{table}

\subsubsection{Dimer abundance sensitivity}
We first estimate the sensitivity $\frac{\partial}{\partial\theta_3}\E[ X_{\theta,D}(t)]$ at time $t=1$, with  $\theta = (200,100,0.1,25,1,1)$, and with zero initial quantities.
In the first bar graph in Figure \ref{fig:dimer}, we show the time required by each  method to compute an estimate to within $5\%$ of the sensitivity value.
The GS Hybrid method is again the most efficient of the unbiased methods, returning the estimate over 8 times faster than PPA and over 600 times faster than the LR+CV method.   In this experiment, for the GS Hybrid method to achieve the target variances determined by the optimization procedure, approximately 53\% of the estimates samples were pathwise estimates, with the other 47\% being coupled likelihood estimates.  See Section \ref{sec:issues}.

The CFD method with $h=(0.1)\theta_3$  is seen to be significantly more efficient than the other methods, including the unbiased methods.  Of course,  the bias of any such finite difference method is generally unknown, which is an issue if high accuracy is a priority.  For example, with $h=(0.1)\theta_3$ the CFD method returns an estimate of $145 \pm 1$, while the actual sensitivity is $\approx$ 141; that is, the bias is approximately $3\%$ of the sensitivity value.  Furthermore, as expected, the variance is inversely proportional to the size of $h$, and when $h$ is changed to $(0.01)\theta_3$, the CFD method becomes less efficient than all other methods except LR+CV.  This illustrates the issue for biased methods that, a priori, one generally does not know which values of $h$ will provide an efficient estimate with acceptable bias. 
The RPD hybrid method suffers a similar difficulty in the choice of $w$:  one generally cannot know the bias of a particular $w$ without numerical experimentation.  For example, with $w=(0.1)t$, the RPD Hybrid method  also has a bias of approximately $3\%$, as it returns an estimate of  $145 \pm 1$.

\begin{figure}
\centering
\textbf{Comparison of efficiency for $\theta_3$ sensitivity estimation, dimerization model}
\vskip1ex
\begin{multicols}{3}
\includegraphics[height=2in]{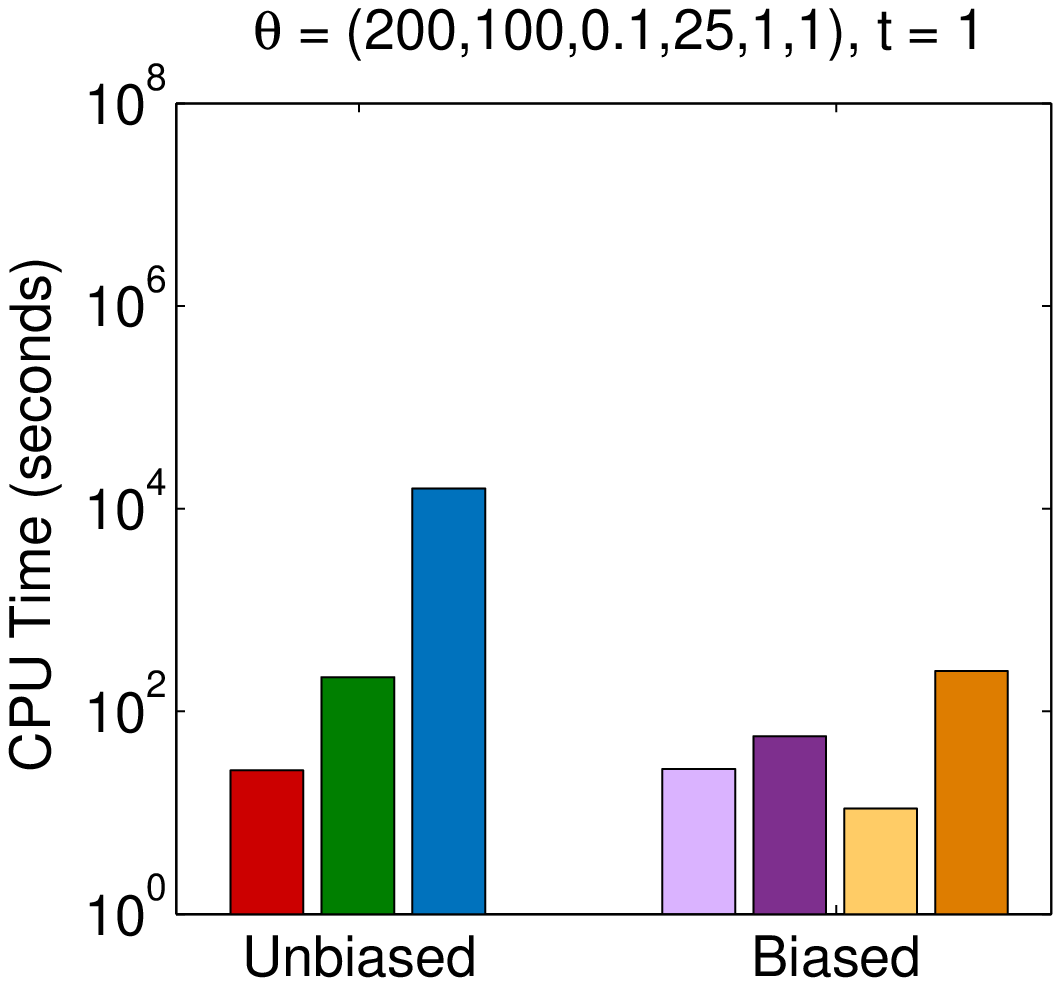}
\columnbreak

\includegraphics[height=2in]{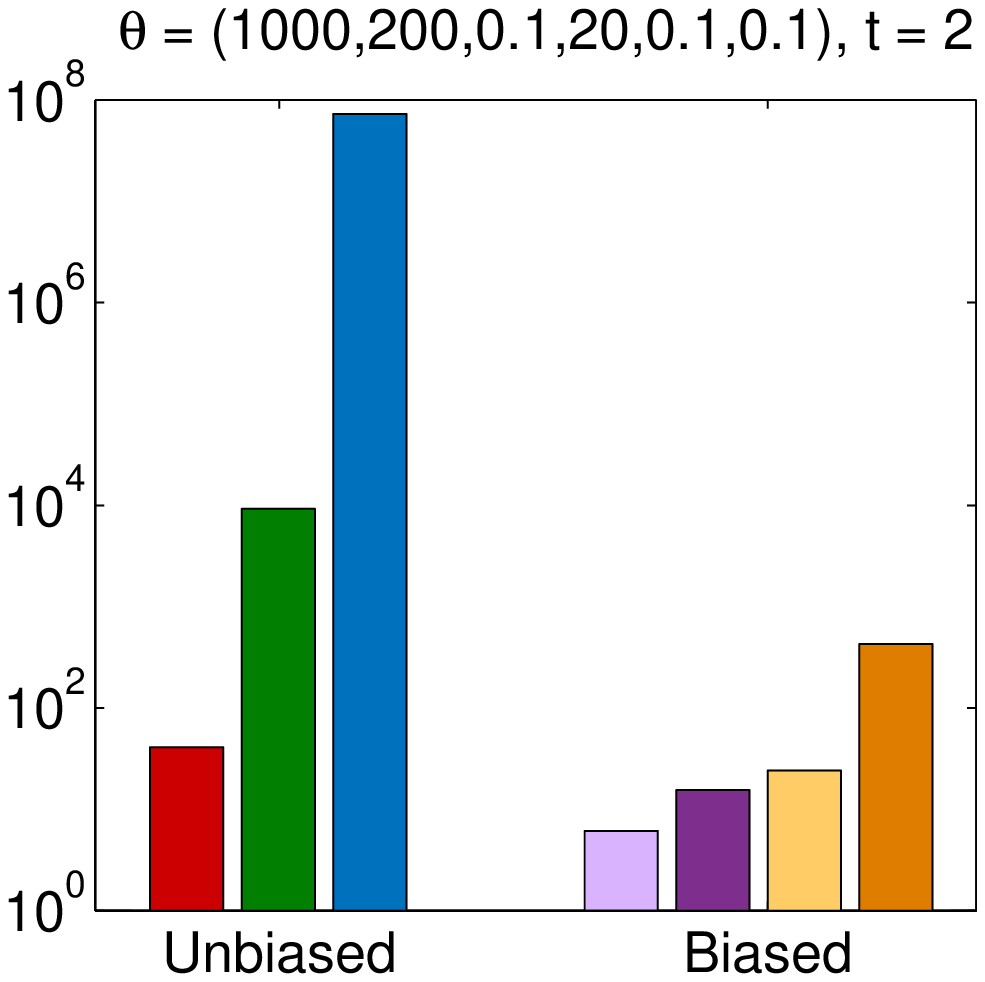}
\columnbreak

$\begin{array}{l}
\\
\includegraphics[height=.75in]{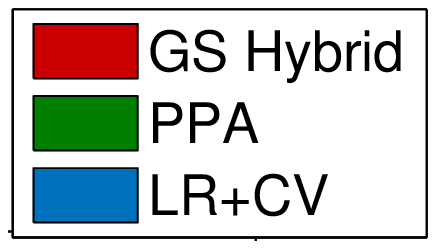}\\
\includegraphics[height=.9in]{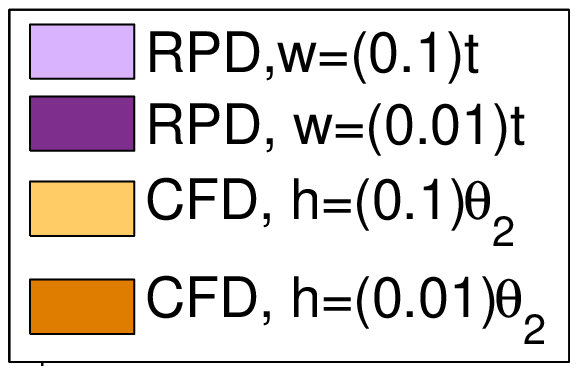}
\end{array}$
\end{multicols}
\caption[1st order method comparison, Dimerization I]{A comparison of efficiency of the sensitivity methods on the dimerization model of Section \ref{sec:ex:dimer} to compute $\frac{\partial}{\partial\theta_3}\E [X_{\theta,D}(t)]$.  We provide two estimates.  The first estimate is at $\theta=(200,100,0.1,25,1,1)$, $t=1$, and zero initial conditions; the second is at  $\theta = (1000,200,0.1,20,0.1,0.1)$, $t = 2$, and an initial condition of $X_{\theta,M}(0)=50$ and other initial abundances equal to 0.  
CPU gives computation time in seconds required to reach a confidence half-width of $5\%$ of the sensitivity value. In the second graph, the CPU time given for the LR+CV method is an estimate based on the variance of partial data.}
\label{fig:dimer}
\end{figure}

We next include results for computing $\frac{\partial}{\partial\theta_3}\E [X_{\theta,D}(t)]$ at a different set of parameters, namely  $\theta = (1000,200,0.1,20,0.1,0.1)$, at time $t = 2$ and with an initial condition of $X_{\theta,M}(0)=50$ and other initial abundances equal to 0.  As shown in the second graph in Figure \ref{fig:dimer}, in order to achieve a half-width of approximately $5\%$ of the value of the sensitivity, the GS Hybrid method is by far the most efficient unbiased method.  In particular, the PPA method requires over 225 times more computation time than the GS hybrid  method. We estimate that the LR+CV method requires approximately $1.8\times10^6$ times more computation time than the GS hybrid method, though we were not able to complete the numerical computations for the LR+CV method due to the fact that the time required to do so was so large.    
We note that, for this example, the approximate paths $Z_\theta$ simulated for the pathwise estimates of the GS Hybrid method were all valid realizations of the original process $X_\theta$. That is, with very high probability, the coupled likelihood estimator is zero or near zero.  Thus, contrary to the previous set of parameters, in this experiment, all estimates were pathwise estimates.  See Section \ref{sec:issues}.

Note that for this particular experiment, the RPD Hybrid method is more efficient than the GS Hybrid method, by a factor of almost 7 when $w=(0.1)t=0.2$, and by a factor of about 2.5 when $w=(0.01)t=0.02$.  
 Furthermore, the bias of the RPD method is less significant than for the previous choice of parameters.  In particular, the bias of the RPD Hybrid method when $w=(0.1)t$ is only approximately $1\%$ of the actual value, returning an estimate of $557\pm1$ while the actual value is $\approx552$; when $w=(0.01)t$, the bias is only about $0.8\%$.  As described above for the GS Hybrid method, the RPD Hybrid method used only pathwise estimates in this experiment. 
Also note that the RPD Hybrid method, with either choice of $w$, is more efficient than the CFD method at either choice of $h$ we considered.

\subsubsection{Integrated dimerization rate sensitivity}\label{sec:flux}

We  consider the functional
\[ \int_0^t \lambda_3(\theta,\xts)ds = \int_0^t \theta_3 X_{\theta,P}(s)(X_{\theta,P}(s)-1)\, ds,
\]
which is the integral of the rate of the dimerization reaction, at $t=5$ and at $\theta_0 = (200,10,0.01,25,1,1)$.
This quantity is a functional of the path and we therefore use the pathwise hybrid method, outlined in and around \eqref{eq:2408957},  on this quantity directly.  That is, we do not need to use the martingale representation \eqref{eq:mart_rep} as we have in previous examples.  The RPD and PPA methods are not applicable for the computation of this sensitivity.
Also note that, unlike in previous examples, the functional depends explicitly on $\theta$, which requires the methods to take into account the partial derivative of the functional in both pathwise and likelihood ratio estimators.

Instead of estimating a single derivative, we estimate the full gradient.  Further, for this example we estimate the efficiency of the methods by simulating each valid method for a fixed amount of time and comparing the resulting confidence intervals for each of the entries of the gradient.  Table \ref{table:flux2} provides this comparison for the pathwise hybrid, the LR+CV, and the CFD methods.  
As shown in the table, the pathwise hybrid method is significantly more precise than the LR+CV method, which is the only other unbiased method that is applicable for this problem. 
The pathwise hybrid method is also significantly more precise than the CFD method, which for this experiment used the relatively large perturbations of $h=(0.1)\theta_i$ for the $i$th entry of the gradient (which leads to a smaller variance). The relatively poor behavior of the CFD method is partially due to the fact that, unlike the pathwise hybrid and LR+CV  methods, the CFD method cannot reuse paths for different gradient estimates since the simulated paths have only one particular parameter perturbed.  This problem with finite difference methods grows in significance as the dimension of $\theta$ grows.

\begin{table}
\centering
\begin{tabular}{|clr p{.25cm}  lp{.01cm}l  rp{.25cm}  lp{.01cm}l  rp{.25cm}ll|}
\hline
& \multicolumn{1}{c}{}&  \multicolumn{3}{c}{Pathwise hybrid} &&& \multicolumn{3}{c}{LR+CV} &&& \multicolumn{3}{c}{CFD} &
\\
\hline
\multirow{6}{*}{\parbox{.5cm}{$\nabla_\theta$}} & & 
0.5713&$\pm$&0.0067  &   &   &   0.5685&$\pm$&0.0501     &&&   0.5669&$\pm$&0.0146  &\\ 
&&11.48&$\pm$&0.13   &   &   &  11.14&$\pm$&0.67      & &&  11.26 &$\pm$&  0.27  &\\

&&3401&$\pm$&34    &  &    & 3162&$\pm$&308        &&&   3403&$\pm$& 126& \\

&&-4.559&$\pm$& 0.051 &    &  &   -5.046&$\pm$&0.419      & &&  -4.544  &$\pm$& 0.114&\\

&&-55.95&$\pm$&0.59  &    &  &   -57.33&$\pm$&4.48    & &&    -53.32  &$\pm$& 1.57 &\\
&&0.0&$\pm$&0.0    &    &  &   -0.1&$\pm$&2.4     &  &&   0.0 &$\pm$& 0.0& \\
\hline
CPU Time  &&& 68 & & &&& 68 &&  &&& 68 &&
\\
\hline
\end{tabular}
\caption[1st order method comparison, Dimerization IV]{A comparison of sensitivity methods on the dimerization model of Section \ref{sec:ex:dimer}.  Estimates of $\nabla_\theta \E [\int_0^t \lambda_5(\theta,\xts)\, ds]$ are given for $t=5$ and at $\theta_0 = (200,10,0.01,25,1,1)$.  CPU gives computation time in seconds.  Recall that the hybrid and LR+CV methods are unbiased, while CFD is not.  Note that the total computation time used by each of the three methods is approximately equal (we have rounded the values to the nearest second for clarity).  As the CFD method must compute each estimate one by one, the total computation time was allocated approximately equally for each of the six estimates.}
\label{table:flux2}

\end{table}

\section{Conclusions}\label{sec:concl}

We have provided a new class of methods for the estimation of parametric sensitivities.  These  hybrid methods include a pathwise estimate but also a correction term, ensuring that the bias is either mitigated (in the case of the RPD hybrid method) or zero.  In particular, the GS hybrid method is, along with the LR and PPA methods, only the third unbiased method so far developed in the current setting for the estimation of derivatives of the form $\ppti \E[f(X_\theta(t))]$.  

For computing sensitivities of the form $\ppt \E[ f(\xtt)]$ at some fixed time $t$, two methods were highlighted. The GS hybrid method is unbiased, and can be significantly more efficient than existing unbiased methods.  At the cost of a small, controllable bias, the RPD hybrid method, which utilizes the RPD method of \cite{Khammash2012} for the pathwise estimate, can often increase efficiency futher, particularly at large times when the system may be nearing stationarity.
A useful avenue of future work will be to study these and other existing sensitivity methods on a wider range of networks and parameter values to better describe which method might be most efficient for a given model of interest.

\vspace{.2in}
\noindent\textbf{Acknowledgments.}
	Anderson and Wolf were both supported by NSF grant  DMS-1318832.   Anderson was also supported under Army Research Office grant W911NF-14-1-0401. We thank  James Rawlings for suggesting the study of Michaelis--Menten kinetics.

	\appendix

\section{Proof of Theorem \ref{thm:pw_first}}\label{app:PE}

We restate Theorem \ref{thm:pw_first}.
\vspace{.1in}

\noindent \textbf{Theorem 1.}\textit{
Suppose that the process $\ztheta$ satisfies the stochastic equation \eqref{eq:pmodel} with $\lambda_k$ satisfying Conditions \ref{nonint} and \ref{cond:extra} on a neighborhood $\Theta$ of $\theta$.  Suppose that the function $F$ satisfies Condition \ref{Fcond} on $\Theta$.  For some $0 \le a \le b < \infty$, let  $L_Z(\theta) = \int_a^b F(\theta,\zts)\, ds$. Then  $ \frac{\partial}{\partial \theta_i} \E \left[ L_Z(\theta)\right]= \E\left[ \frac{\partial}{\partial \theta_i} L_Z(\theta)\right]$, for all $i \in \{1,\dots,R\}$.
}

\vspace{.1in}

The proof of Theorem \ref{thm:pw_first} is similar to that of Theorem 5.1 in \cite{Glasserman1990}.  The main difference is in the proof of the continuity of the function $L$, which is our Lemma \ref{lem:ascont} below.  As in Section \ref{sec:pwest}, for convenience throughout this appendix we take $R=1$ (so that $\theta$ is 1-dimensional).
\vspace{.1in}

 We first need some preliminary results.   Let $N(\theta,t)$ be the number of jumps of $\ztheta$ through time $t$.

\begin{lem}\label{moments}
For any fixed and finite $t$, $q\in[1,\infty)$, and $c\in[1,\infty)$, we have
\[ \E \left[\sup_{\theta\in\Theta}N(\theta,t)^q\right] < \infty \;,\quad \E\left[ \sup_{\theta\in\Theta}\sup_{s\in[0,t]}\Vert \zts \Vert^q\right] < \infty
\quad \textrm{ and } \quad
\mathbb{E}\left[\displaystyle\sup_{\theta\in\Theta}c^{N(\theta,t)}\right] < \infty.
\]
\end{lem}
 
\begin{proof}
Note that by  Condition  \ref{cond:extra}, $N(\theta,t)$ is stochastically bounded, uniformly in $\theta$, by a Poisson random variable $\hat N$ with parameter $\tilde\Gamma = tK\Gamma_M$.  This proves the first bound immediately.  To see the second result, note that $\sup_{s\in[0,t]}\Vert \zts \Vert \leq \Vert Z_\theta(0)\Vert + N(\theta,t)\max_{k}| \1 \cdot \zeta_k|$ and use the first result.  To prove the final bound, use that $\E \left[\sup_{\theta\in\Theta}c^{ N(\theta,t)}\right] \leq \E [c^{\hat N}]$, and that
\begin{align*}
\E [c^{\hat N}] = \sum_{m=0}^\infty c^m \mathbb{P}(\hat N = m) = \sum_{m=0}^\infty c^m \frac{\tilde\Gamma^m}{m!} e^{-\tilde\Gamma} = e^{-\tilde\Gamma}\sum_{m=0}^\infty  \frac{(c\tilde\Gamma)^m}{m!} = e^{-\tilde\Gamma} e^{c\tilde\Gamma} < \infty.
\end{align*}
\end{proof}

\begin{lem}\label{lem:ascont}
For any $\theta\in\Theta$ and for $h>0$ such that $(\theta-h,\theta+h)\subset\Theta$, with probability $1-O(h^2)$ we have that $L_Z(\theta)$ is continuous and piecewise differentiable on $(\theta-h,\theta+h)$.
\end{lem}

\begin{proof}
There are two parts to the proof.  First, we show that if on the interval $(\theta-h,\theta+h)$ no more than one change occurs to the embedded chain $\hat Z_\ell$ on the interval $[a,b]$, then $L_Z(\theta)$ is continuous on that interval.  Second, we require that the probability of two or more such changes is $O(h^2)$.  The proof of the second claim follows as in the second part of Appendix 5.B in \cite{Glasserman1990}, p. 120, so we do not include it here.

We prove the first claim.  Suppose that there is at most one change to the embedded chain in the time interval $[a,b]$ on $(\theta-h,\theta+h)$.  Then one of the following cases occurs:
\begin{enumerate}[(i)]
\item there is no change to the embedded chain,
\item two (or more) jumps switch order through time $b$, causing a change in the embedded chain of $Z_\theta$, or
\item some jump enters or exits the interval $[a,b]$, changing the number states appearing in the integral $L_Z$.
\end{enumerate}
We have crucially used the non-interruptive Condition \ref{nonint} here, and the fact that $Z_\theta$ satisfies the random time change representation \eqref{eq:pmodel}, to exclude any other possibilities, including interruptions.
What we must show is that $L_Z$ is continuous in each case.
Recall from \eqref{L} that
\begin{equation}\label{pfl}
L_Z(\theta) = \sum_{\ell=0}^{N(\theta,b)} F(\theta,\hat Z_\ell(\theta))[ T_{\ell+1}^{\theta}\wedge b -T_{\ell}^{\theta}\vee a ]^+
\end{equation}
and that $F$ is continuous in $\theta$ by assumption.  By work in Section \ref{sec:pwest}, the jump times $T_\ell^\theta$ are continuous except possibly at values of $\theta$ at which the embedded chain of $Z_\theta$ changes.  Thus it is clear that $L_Z$ is continuous in case ($i$).


Now suppose that ($ii$) occurs at some point $\theta^*\in(\theta-h,\theta+h)$.  Then two reactions $k$ and $m$ occur at the same time.  (The case when three or more reactions occur simultaneously is essentially the same.)  Further suppose these reactions occur as the $\ell\th$ and $(\ell+1)\st$ jumps.  
Then at $\theta^*$, there is a discontinuity in $\hat Z_\ell(\theta)$: from one side the limit is $\hat Z_{\ell-1}(\theta) + \zeta_k$ and from the other it is $\hat Z_{\ell-1}(\theta) + \zeta_m$.  However,  by the non-interruptive Condition, the two reactions may occur in either order, and the net result of the two reactions is the same regardless: $\zeta_k + \zeta_m$ is added to the system. That is, $\hat X_{\ell+1}(\theta)\equiv Z_{\ell-1}(\theta)+\zeta_k + \zeta_m$ on the whole interval, and furthermore, this crossover of jumps affects no other states of the embedded chain.

Then in the summation \eqref{pfl}, any given term changes continuously except possibly the $\ell\th$ term,
\begin{equation}\label{sumF}
 F(\theta,\hat Z_\ell(\theta))[ T_{\ell+1}^{\theta}\wedge b -T_{\ell}^{\theta}\vee a ]^+ .
\end{equation}
But at $\theta^*$, we have $T_{\ell+1}^{\theta^*}=T_{\ell}^{\theta^*}$.  That is, neither reaction is postponed because the intensities of both are strictly positive.
Therefore, the term \eqref{sumF} is zero at the point of discontinuity, 
and $L_Z(\theta)$ is continuous at $\theta^*$ as needed.

Suppose instead that at $\theta^*$ case ($iii$) occurs.  
Since an additional jump time appears in the interval $[a,b]$ at $\theta^*$, an additional term may show up in the summation \eqref{pfl}.  However, this new jump time $T_\ell^\theta$ must be equal to either $a$ or $b$.  Then $[ T_{\ell+1}^{\theta}\wedge b -T_{\ell}^{\theta}\vee a ]^+$ is zero, and $L_Z$ is again continuous at $\theta^*$.

Finally, $L_Z$ is piecewise differentiable in each case.  Indeed, by the derivations in Section \ref{sec:pwest}, $L_Z$ is differentiable except possibly at values of $\theta$ at which the embedded chain changes, and by assumption there is at most one such value.
\end{proof}


We now prove two useful bounds before finally giving the proof of Theorem \ref{thm:pw_first}.   For the remainder, we assume for convenience that $\Gamma_M, \Gamma_m$, and $\Gamma'$ are at least 1.  

\begin{lem}\label{lem:ddelbound}
For each $\ell$ from 0 to $N(\theta,b)$ we have
\[ M_\ell : = \max_k\max_{j\leq \ell}\bigg| \ppt S_k^\theta(T_j^\theta)\bigg| \leq \Gamma' b(2\Gamma_M\Gamma_m)^\ell, \]
where $\Gamma_M, \Gamma_m,$ and $\Gamma'$ are as in Condition \ref{cond:extra}.   
\end{lem}

\begin{proof}
Consider \eqref{pptDelta} and \eqref{pptT} and recall that for each $k$ we have $\ppt S^\theta_{k}(T_{0}^\theta)=0$.  Then
\[
\bigg|\ppt \Delta^\theta_0 \bigg|= \bigg| \frac{\Delta_0^\theta}{ \lambda_{k_\ell}(\theta,\hat Z_\theta(0))}
\ppt \lambda_{k^0}(\theta,\hat Z_\theta(0)) \bigg| \leq \Delta_0^\theta \Gamma' \Gamma_m.
\]
Then for any $k$, we have 
\[\ppt S^\theta_{k}(T_{1}^\theta)= \Delta^\theta_0\ppt \lambda_{k}(\theta,\hat Z_\theta(0)) + \lambda_{k}(\theta,\hat Z_\theta(0))\ppt \Delta^\theta_0,\]
so that
\[M_1 = \max_k\bigg| \ppt S_k^\theta(T_1^\theta)\bigg| \leq \Delta^\theta_0\Gamma' + \Gamma_M\Delta_0^\theta \Gamma' \Gamma_m \leq 2\Gamma'\Gamma_m\Gamma_M\Delta_0^\theta.\]

\noindent Similarly, for a given $\ell$ we have
\begin{align*}\begin{split}
\bigg| \frac{\partial}{\partial\theta}\Delta_\ell^\theta \bigg|
& \leq \bigg| \frac{\Delta_\ell^\theta}{ \lambda_{k_\ell}(\theta,\hat Z_\ell(\theta))}
\ppt \lambda_{k_\ell}(\theta,\hat Z_\ell(\theta))\bigg| + \bigg|\lambda_{k_\ell}(\theta,\hat Z_\ell(\theta))^{-1} \ppt S^\theta_{k_\ell}(T_\ell^\theta)\bigg|
\\
& \leq \Delta_\ell^\theta \Gamma' \Gamma_m +  \Gamma_m M_{\ell-1}.
\end{split}\end{align*}
Therefore, using that
\[\ppt S^\theta_{k}(T_{\ell}^\theta)= \ppt S^\theta_k(T_{\ell-1}^\theta) + \Delta^\theta_{\ell-1}\ppt \lambda_{k}(\theta,\hat Z_{\ell-1}(\theta)) + \lambda_{k}(\theta,\hat Z_{\ell-1}(\theta))\ppt \Delta^\theta_{\ell-1}\]
and noticing that the $M_\ell$ are nondecreasing, we see that
\begin{align*}
M_\ell &\leq M_{\ell-1} + \Gamma'\Delta^\theta_{\ell-1} + \Gamma_M\bigg|\ppt \Delta^\theta_{\ell-1}\bigg|\\
& \leq  M_{\ell-1} + \Gamma'\Delta^\theta_{\ell-1} + \Gamma_M(\Delta_{\ell-1}^\theta \Gamma' \Gamma_m +  \Gamma_m M_{\ell-2})
\\
& \leq  M_{\ell-1} + \Gamma'\Delta^\theta_{\ell-1} + \Gamma_M(\Delta_{\ell-1}^\theta \Gamma' \Gamma_m +  \Gamma_m M_{\ell-1})
\\
& \leq 2\Gamma_M\Gamma_m M_{\ell-1} +  2\Gamma'\Gamma_M\Gamma_m\Delta^\theta_{\ell-1}.
\end{align*}
Iterating this inequality, we see that
\[ M_\ell \leq (2\Gamma_M\Gamma_m)^{\ell-1}2\Gamma'\Gamma_M\Gamma_m \sum_{j=0}^{\ell-1} \Delta^\theta_{j} \leq  \Gamma'b (2\Gamma_M\Gamma_m)^{\ell}.\qedhere
\]
\end{proof}

\begin{cor}
For each $\ell$ from 0 to $N(\theta,b)$ we have
\[ \bigg|\ppt\Delta_\ell^\theta \bigg| \leq 2\Gamma'b\Gamma_m(2\Gamma_M\Gamma_m)^\ell,
 \]
where $\Gamma_M, \Gamma_m,$ and $\Gamma'$ are as in Condition \ref{cond:extra}.
\end{cor}

\begin{proof}
By \eqref{pptDelta}, the two final assumptions on $Z_\theta$ from Appendix A, and Lemma \ref{lem:ddelbound}, we have that 
\begin{align*}\begin{split}
\bigg|\frac{\partial}{\partial\theta}\Delta_\ell^\theta \bigg|
&\leq  \bigg| \frac{\Delta_\ell^\theta}{ \lambda_{k_\ell}(\theta,\hat Z_\ell(\theta))}
\ppt \lambda_{k_\ell}(\theta,\hat Z_\ell(\theta)) \bigg| + \bigg| \lambda_{k_\ell}(\theta,\hat Z_\ell(\theta))^{-1} \ppt S^\theta_{k_\ell}(T_\ell^\theta)\bigg|
\\
&\leq  b \Gamma_m \Gamma' +  \Gamma_m \bigg|\ppt S^\theta_{k_\ell}(T_\ell^\theta)\bigg|
\\
& \leq b \Gamma_m \Gamma' +  \Gamma_m \Gamma'b(2\Gamma_M\Gamma_m)^\ell
\\
& \leq 2\Gamma'b\Gamma_m(2\Gamma_M\Gamma_m)^\ell.
\end{split}\end{align*}
\end{proof}


We finally turn to the proof of Theorem \ref{thm:pw_first}.  As noted previously, the proof of the theorem now follows similarly to the proof of Theorem 5.1 in \cite{Glasserman1990}.
\vspace{.125in}

\noindent\textit{Proof of Theorem \ref{thm:pw_first}.} Let $\tilde h$ be the infimum over $h$ for which two or more changes occur to the embedded chain of $Z_\theta$ through $(\theta-h,\theta+ h)$ on the time interval $[a,b]$. That is, $\tilde h$ is the \textit{second} place at which a change in the embedded chain occurs.  Note that $\tilde h> 0$ is positive with probability 1.  Without loss of generality, $(\theta-\tilde h, \theta+\tilde h)\subset \Theta$.  We must prove the middle equality in
\begin{align*} 
\ddt \E [L_Z(\theta)] &= \lim_{h\to 0}\E [h^{-1}[L_Z(\theta+h)-L_Z(\theta)]] 
= \E \left[\lim_{h\to 0}h^{-1}[L_Z(\theta+h)-L_Z(\theta)]\right] 
= \E \left[\ddt L_Z(\theta)\right] .
\end{align*}
We write 
\begin{align}\begin{split} \label{thmsplit}
\E [h^{-1}[L_Z&(\theta+h)-L_Z(\theta)]] \\= &\; \E [h^{-1}[L_Z(\theta+h)-L_Z(\theta)] \1(h < \tilde h)] + \E [h^{-1}[L_Z(\theta+h)-L_Z(\theta)] \1(h \geq \tilde h)].
\end{split}\end{align}

\noindent Consider the first term.  By Lemma \ref{lem:ascont}, and since by the definition of $\tilde h$ at most one change occurs to the embedded chain for $h<\tilde h$,  we have that $L_Z$ is  continuous and  piecewise differentiable on $(\theta-\tilde h, \theta+\tilde h)$. By a generalized mean value theorem (e.g. \cite{Dieudonne}),
\[
\big|h^{-1}[L_Z(\theta+h)-L_Z(\theta)]\1(h < \tilde h)\big| \leq \sup_{\theta\in\Theta}\bigg|\ddt L_Z(\theta)\bigg|,
\]
where the supremum is over those points where the derivative exists.  We will show that this supremum has finite expectation; therefore, since as $h\to 0$,
\[
	h^{-1}[L_Z(\theta+h)-L_Z(\theta)]\1(h < \tilde h)\overset{a.s.}{\to}\ddt L_Z(\theta)
\]
we will have by the dominated convergence theorem that $\E[h^{-1}[L_Z(\theta+h)-L_Z(\theta)]\1(h < \tilde h)]\to\E[\ddt L_Z(\theta)]$.  We will also show that the second term in \eqref{thmsplit} goes to zero as $h\to 0$, which proves the theorem.

Write $N:=N(\theta,b)$ and recall that
\begin{align*}
\bigg|\ddt &L_Z(\theta)\bigg|  = \bigg| \sum_{\ell=0}^N [  T_{\ell+1}(\theta)\wedge b -T_{\ell}(\theta)\vee a ]^+
 \left(\ppt F(\theta,\hat Z_\ell(\theta)) \right) 
 +  F(\theta,\hat Z_\ell(\theta))\ppt[  T_{\ell+1}(\theta)\wedge b -T_{\ell}(\theta)\vee a ]^+
\bigg|
\\
& \leq \bigg| \sum_{\ell=0}^N [  T_{\ell+1}(\theta)\wedge b -T_{\ell}(\theta)\vee a ]^+
 \left(\ppt F(\theta,\hat Z_\ell(\theta)) \right)\bigg| 
 + \bigg|  \sum_{\ell=0}^N F(\theta,\hat Z_\ell(\theta))\ppt[  T_{\ell+1}(\theta)\wedge b -T_{\ell}(\theta)\vee a ]^+
\bigg|.
\end{align*}
We now consider these two terms separately. By Condition \ref{Fcond} on $F$, 
\begin{align*}
\bigg| \sum_{\ell=0}^N [  T_{\ell+1}(\theta)\wedge b -T_{\ell}(\theta)\vee a ]^+
 \left(\ppt F(\theta,\hat Z_\ell(\theta)) \right)\bigg|
  & 
  \leq \sum_{\ell=0}^N  T_{\ell+1}(\theta)\wedge b -T_{\ell}(\theta)\vee a ]^+
 \bigg|\ppt F(\theta,\hat Z_\ell(\theta)) \bigg|
 \\ & \leq
 C_2\sum_{\ell=0}^N [  T_{\ell+1}(\theta)\wedge b -T_{\ell}(\theta)\vee a ]^+(1+\Vert \hat Z_\ell^\theta\Vert^{c_2})
 \\
& \leq 
C_2  (b-a)(1+ \max_{\ell\leq N}\Vert\hat Z_\ell^\theta\Vert^{c_2})
  \\
& \leq C_2 (b-a)(1+ \sup_{\theta\in\Theta}\sup_{s\in[0,b]}\Vert \zts\Vert^{c_2}).
\end{align*}

\noindent Now, from \eqref{derivs} and our work in Lemma \ref{lem:ddelbound} we have for any $\ell$ that 
\[
\bigg| \ppt[  T_{\ell+1}(\theta)\wedge b -T_{\ell}(\theta)\vee a ]^+ \bigg|
\leq 
\sum_{j=0}^{N} \bigg| \ppt\Delta_j \bigg|.
\]

\noindent Therefore, for the second term, 
\begin{align*}
 \bigg|\sum_{\ell=0}^N F(\theta,\hat Z_\ell(\theta))\ppt[  T_{\ell+1}(\theta)\wedge b -T_{\ell}(\theta)\vee a ]^+ \bigg|
& \leq 
 C_1 \sum_{\ell=0}^N (1+\Vert \hat Z_\ell^\theta\Vert^{c_1})
 \bigg|\ppt[  T_{\ell+1}(\theta)\wedge b -T_{\ell}(\theta)\vee a ]^+
\bigg| 
 \\
& \leq  
 C_1 (1+ \max_{\ell\leq N}\Vert\hat Z_\ell^\theta\Vert^{c_1}) \sum_{\ell=0}^N 
\sum_{j=0}^{N} \bigg| \ppt\Delta_j \bigg|
\\
& \leq  
 C_1 (1+ \max_{\ell\leq N}\Vert\hat Z_\ell^\theta\Vert^{c_1}) \sum_{\ell=0}^N \sum_{j=0}^{N}
2\Gamma'T\Gamma_m^2(2\Gamma_M\Gamma_m)^j
  \\
& \leq  
 C_1 (1+\sup_{\theta\in\Theta}\sup_{s\in[0,b]}\Vert \zts\Vert^{c_1}) N^2 2\Gamma'T\Gamma_m(2\Gamma_M\Gamma_m)^N.
\end{align*}

\noindent By Lemma \ref{moments} and repeated applications of the Cauchy-Schwarz inequality, we see that both of the bounds we have computed are bounded uniformly in $\theta$ on $\Theta$ by a quantity of finite expectation as needed.

Finally, we must show that $\E [h^{-1}[L_Z(\theta+h)-L_Z(\theta)] \1(h \geq \tilde h)]$ goes to zero as $h\to 0$.  By using the Cauchy-Schwarz inequality, we see that
\[
\E \left[h^{-1}[L_Z(\theta+h)-L_Z(\theta)] \1(h \geq \tilde h)\right]^2
\leq 
h^{-2}\E \big[ [L_Z(\theta+h)-L_Z(\theta)]^2\big] P(h \geq \tilde h).
\]
Since by Lemma \ref{lem:ascont} we have $P(h \geq \tilde h) = O(h^2)$, and since  $[L_Z(\theta+h)-L_Z(\theta)] \overset{a.s}{\to} 0$, we are done by the dominated convergence theorem if we can show that $[L_Z(\theta+h)-L_Z(\theta)]^2$ is bounded by an integrable function.
By Condition \ref{Fcond} on $F$, for any $\theta\in\Theta$,
\begin{align}\begin{split}\label{eq:l2bound}
[L_Z(\theta)]^2 &= \left(\int_a^bF(\theta,\zts)ds\right)^2 \leq(b-a) \int_a^b \big( F(\theta,\zts) \big)^2 ds
\\ 
& \leq (b-a) \int_a^b C_1^2(1+ \Vert \zts\Vert^{c_1})^2 ds
\\
& \leq C_1^2(b-a)^2(2 + 2\sup_{\theta\in\Theta}\sup_{s\in[0,b]}\Vert \zts\Vert^{2c_1}),
\end{split}
\end{align}
where the final line follows because $(a+b)^2 \leq 2a^2 + 2b^2$.  This bound has finite expectation by Lemma \ref{moments}, and is also uniform, so that it holds for $|L_Z(\theta+h)^2|$ as well.  Then 
as needed,
\[
|L_Z(\theta+h)-L_Z(\theta)|^2 \leq 2[L_Z(\theta+h)]^2+2[L_Z(\theta)]^2 \leq 4\sup_{\theta\in\Theta}[ L_Z(\theta)]^2,
\]
which has finite expectation by taking the supremum of \eqref{eq:l2bound}.
\hfill $\square$

\bibliographystyle{amsplain}

\bibliography{refs}  
\end{document}